\documentclass[11pt]{scrartcl}

\usepackage[english]{babel}
\usepackage[T1]{fontenc}
\usepackage[utf8]{inputenc}

\usepackage{amsmath}
\usepackage{mathtools}
\usepackage{amssymb}
\usepackage{amsthm}
\theoremstyle{definition}
\newtheorem{definition}{Definition}[section]
\newtheorem{remark}[definition]{Remark}
\theoremstyle{plain}
\newtheorem{lemma}[definition]{Lemma}
\newtheorem{theorem}[definition]{Theorem}
\newtheorem{proposition}[definition]{Proposition}

\usepackage{bm}

\usepackage{graphicx}
 \graphicspath{ {./Plots/} }

\usepackage{enumerate}

\usepackage{url}

\usepackage{here}

\usepackage{lmodern}

\usepackage{fancyvrb}

\usepackage[plainpages=false]{hyperref}
\hypersetup{
    colorlinks=true,
    linkcolor=cyan,
    citecolor = cyan
    }

\usepackage{longtable}

\usepackage{listings}
\usepackage{multicol}
\usepackage[section]{algorithm}
\usepackage{algpseudocode}

\usepackage[style=alphabetic,backend=bibtex,maxbibnames=99]{biblatex}
\usepackage{csquotes}

\DeclareMathOperator*{\argmin}{argmin}
\DeclareMathOperator*{\argmax}{argmax}
\newcommand{\x}{\bm{\mathrm{x}}}
\newcommand{\X}{\bm{\mathrm{X}}}
\newcommand{\y}{\bm{\mathrm{y}}}
\newcommand{\e}[2]{{#1}\mathrm{e}{#2}}

\counterwithin{figure}{section}
\numberwithin{equation}{section}
\numberwithin{table}{section}

\begin{document}
\counterwithin{lstlisting}{section}

\pagenumbering{Alph}

\title{{Bachelor thesis} \\[18pt] \Huge Introduction: Swarm-based gradient descent for non convex optimization \\[18pt]}
\author{Janina Tikko}
\date{18.09.2023}

\maketitle

\begin{center}
\begin{Large}
Department of Mathematics and Computer Science \\[3pt]
Division of Mathematics \\[3pt]
University of Cologne \\[2cm]
\begin{figure}[ht]
	\centering
\includegraphics[scale=0.4]{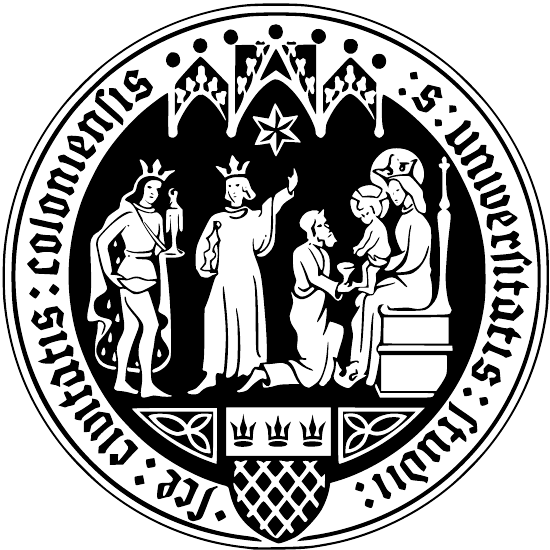}
 \\[2cm]
\end{figure}
Supervisor: Prof. Dr. Angela Kunoth
\end{Large}
\end{center}

\thispagestyle{empty}


\newpage
\subsection*{Acknowledgements}
\thispagestyle{empty}

At this point, I would like to thank Prof. Angela Kunoth for the opportunity to write a bachelor thesis under her supervision. The topic she gave me is based on the work of Prof. Eitan Tadmor and was fascinating. Moreover she encouraged me to contact Prof. Eitan Tadmor directly to help answer my questions. Therefore I would like to thank him for our correspondence. For tips and comments I am also thankful to Sarah Knoll, who I could also turn to at any time.
I would also like to thank my family, my partner and a friend for proofreading.


\newpage
\tableofcontents

\pagenumbering{roman}


\pagenumbering{arabic}

\newpage
\section{Introduction}
The field of optimization has the goal to find an optimal solution to a target function. This means to minimize (or maximize) the target function.
Such optimization problems are found in several scientific disciplines, for example in physics or computer science.
Often it is not possible to find the analytical solution, thus one has to consider numerical approaches.
\\\\
We consider a general optimization problem
\begin{align}
	\argmin_{\x\in \Omega \subset\mathbb{R}^d} F(\x)
	\label{eg:1.1}
\end{align}
for a function $F:\Omega \subset\mathbb{R}^d \to \mathbb{R}$.
 To find a minimum, we can use the classical gradient descent method. First, we compute the local gradient $\nabla F(\x)$ to find a search direction and after choosing a step size, we can run along the function to search for a minimum. Once the gradient approaches zero, we know that a minimum has been found. However, initially we can only assume a local minimum. If we want to find a global minimum, this method is often not suitable. For example the function
 $S(x)=(1.5t-2)^2 \cdot \cos(30\pi+(3\pi t)^2)$ in figure \ref{fig:signal} has many local minima and one global minimum.

 \begin{figure}[ht]
 	\centering
 \includegraphics[scale=0.33]{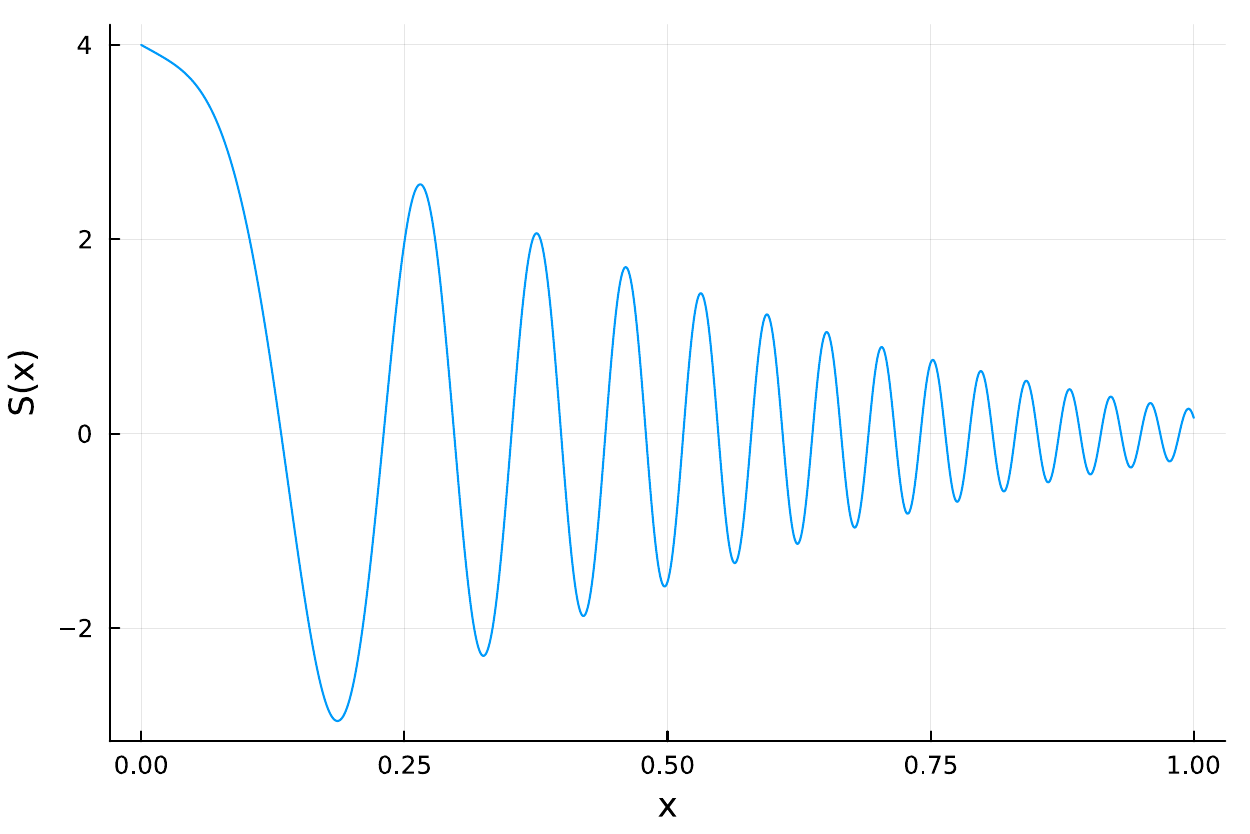}
 \caption{
 Graph of $S(x)=(1.5t-2)^2 \cdot \cos(30\pi+(3\pi t)^2)$.
 }
 \label{fig:signal}
 \end{figure}
 If the starting point is not well chosen, the classical gradient descent method would stop at a local minimum because the gradient equals zero and therefore the method would never reach the global minimum. In the case of the function $S$, a starting position outside of the first quarter of the interval would lead to an incorrect result. Therefore, the subject of this paper is the introduction of the swarm-based gradient descent for solving the global optimization problem \eqref{eg:1.1} based on \cite{tad}.
\\\\
As implied by the name, a swarm of multiple agents is used for finding a global minimum. The agents are characterized by a time-dependent position $\x_i(t^n)\in\Omega \subset\mathbb{R}^d$ and a relative mass $\tilde{m}_i(t^n)\in[0,1]$. Each agent is also given its own step size, defined by its relative mass. Heavier agents receive smaller step sizes and converge to local minima, while lighter agents have a larger step size, improving the global position of the swarm.
As animals in nature, the agents in a swarm communicate with each other. This communication leads to a mass transition between the agents, so that lighter
agents have a possibility to grow heavier and therefore converge to a potential global minima. Another characteristic of this method is the
\glqq Survival of the fittest\grqq{} approach. After each iteration the \glqq worst\grqq agent will be eliminated from the swarm.
\\\\
A further explanation of the algorithm with more detail is given in chapter 2.
In chapter 3, I will also demonstrate the method with an example. Thereby the advantages compared to the gradient descent become clearer.
After that, in chapter 4, I show my implemented version in the programming language Julia, before finally in chapter 5 the convergence and error analysis follows.

\subsection*{The idea behind the communication}
As mentioned before, the classical gradient descent is often the first approach to find a solution for a problem like \eqref{eg:1.1}.
And because of the disadvantages, one might start to try and improve this method.
Naturally the first thing to do would be to consider more than one explorer. Like in daily life, a group is often faster in solving a problem than one individual.
Moreover, a group could be spread around the target function to help avoiding being trapped in local minima. In section \ref{section 3.2} we will see that this is
an improvement, but the method still has some disadvantages.
\\\\
Since we have a group of explorers now, the question is how to improve the group's behavior. For this purpose, we can consider nature's principles.
In nature there are several swarms of animals that act in groups in order to survive. The question is now, what is the characteristic of a swarm?
\\
A swarm is based on both a number of individuals, also called agents, and an interaction process between them. To describe this process, we consider the Cucker-Smale \cite{cuc} model. It describes a pairwise interaction between agents, that steer the swarm towards average heading. The interactions are dictated by a communication kernel. That means, to improve our group of agents, a design of a communication kernel is needed. This leads us to the swarm-based gradient descent method.


\section{Algorithm Swarm-based gradient descent}
The swarm-based gradient method (SBGD) involves three main aspects:
the agents, the step size protocol, and the communication. For determining the step size, we will use the backtracking method \ref{section 2.3} as explained in \cite{tad}.

\subsection{The agents}

The algorithm uses $J\in\mathbb{N}$ agents from $\mathbb{R}^d \times (0,1]$. For $i=1,...,J$ each agent is characterized by a position $\x_i(t)\in\mathbb{R}^d$
and a mass $m_i(t)\in (0,1]$.
The total mass of all agents is constant at all times
\begin{align*}
	\sum_{i\in J} m_i(t) = 1 .
\end{align*}

\subsection{The step size protocol}

In each iteration the agents positions are dynamically adjusted by a time step $h_i$ in direction of the gradient $\nabla F(\x_i(t))$
\begin{align*}
	\frac{\mathrm{d}}{\mathrm{d}t} \x_i(t) = -h_i \nabla F(\x_i(t)).
\end{align*}
The time step $h_i$ depends on the current position $\x_i(t)$ and the relative mass of the agent, where the relative mass of the agent is defined as
\begin{align*}
	\tilde{m}_i(t):=\frac{m_i(t)}{m_+},
\end{align*}
with $m_+:=\max\limits_{i\in J} m_i(t)$. The function $h_i$ should therefore be chosen as a decreasing function of the relative masses, i.e. heavier agents get smaller time steps while lighter agents receive larger step sizes.
\\\\
\begin{remark}
The relative mass $\tilde{m}_i$ can alternatively be understood as the probability of the agent to find a global minimum. Agents with $m_i(t) \ll m_+(t)$ have to get larger step sizes, because at the current position the probability to find a minimum is rather low.
\end{remark}

\subsection{Backtracking}
\label{section 2.3}

To compute the time step $h_i$ we use the \emph{backtracking line search} method based on the Wolfe conditions \cite{wol}.
In each iteration we want to take a time step in direction of the gradient $\nabla F(\x(t))$, thus
\begin{align*}
	\x^{n+1}=\x^n -h\nabla F(\x^n).
\end{align*}
The idea of the \emph{backtracking line search} method is to choose a time step $h$ in such a way, that
\begin{align}
	F(\x^{n+1}(h)) \leq F(\x^n)-\lambda h \vert \nabla F(\x^n)\vert ^2, \hspace{1cm} \lambda\in (0,1).
	\label{eg:2.1}
\end{align}
Of course for $h \ll 1$ is
\begin{align*}
	F(\x^{n+1}(h))=F(\x^n)- h \vert \nabla F(\x^n)\vert ^2 + \{Terms \ of \ higher \ order\}
\end{align*}
and therefore \eqref{eg:2.1} holds for any fixed $\lambda\in (0,1)$. However, this is not what we want and numerically not useful.
We want the step size $h$ as large as possible, so that we can maximize the descent towards $\lambda h \vert \nabla F(\x^n)\vert ^2$.
Therefore, we start with a large $h$ such that
\begin{align*}
	F(\x^{n}-h\nabla F(\x^n)) > F(\x^n)-\lambda h \vert \nabla F(\x^n)\vert ^2
\end{align*}
holds.
Then we successively decrease the step size with a shrinking factor $\gamma>0$ until \eqref{eg:2.1} is reached for $h=h(\x^{n},\lambda)$.
\\\\
Now the choice of $\lambda$ can be problematic. For larger $\lambda$ the step size is limited and no larger jumps are possible.
If $\lambda \ll 1$ there is a danger of taking too small steps and stopping at a local minimum. To get around this, we use the relative mass of the
individual agents to adjust $\lambda$. To do this, define $\psi_q (\tilde{m}^{n+1}):=(\tilde{m}^{n+1})^q$, with $q>0$.
For $i=1,...,J$ we thus obtain the step size
\begin{align}
	h_i^n=h(\x_i^n,\lambda\psi_q(\tilde{m}_i^{n+1})).
	\label{eg:2.2}
\end{align}
The parameter $q$ determines the influence of the relative mass. That means if $q$ is larger, $\psi_q(\tilde{m}_i^{n+1})$ will be smaller and agents with
a relative mass in the middle range can get larger time steps from the backtracking method. By default, we assume $q=1$.

\subsection{The communication}

Communication is the aspect in which the SBGD method differs from others. Considering the relative heights of each agent, we redistribute the masses in each iteration step. The lowest positioned agent attracts the mass from the other agents to approach a potential minimum through smaller step sizes. Meanwhile the other
agents become lighter and lighter and thus explore further in the region of interest.
But a lighter agent may be better positioned after a large time step and thus become the new heaviest agent and therefore approaches a new potential minimum.
This mass transition is described as follows:
\\\\
Set $F_{\max}(t):=\max\limits_{i\in J} F(\x_j(t))$ and $F_{\min}(t):=\min\limits_{i\in J} F(\x_j(t))$. At time $t$, $F_{\max}(t)$ is the maximum height and $F_{\min}(t)$ is the minimum height of the swarm. Thus, we define the relative height of an agent as
\begin{align}
    \label{eg:relativeheights}
	\eta_i(t):=\frac{F(\x_i(t))-F_{\min}(t)}{F_{\max}(t)-F_{\min}(t)} > 0.
\end{align}
With $i^*:=\argmin\limits_{i\in J} F(\x_i(t))$ we then describe the mass transition by
\begin{align*}
\begin{dcases}
	\frac{\mathrm{d}}{\mathrm{d}t} m_i(t) = - \phi_p(\eta_i(t))m_i(t) & i\neq i^* \\
	m_i(t) = 1- \sum\limits_{j\neq i^*} m_j(t) & i=i^* \ ,
\end{dcases}
\end{align*}
where $\phi_p(\eta) = \eta^p \in (0,1]$ and $p>0$. By default $p=1$, but it can be adjusted for optimization purposes.

\subsection{Time discretization}

For the time discretization, we set $t^{n+1}=t^n + \Delta t$, with $\Delta t =1$. Thus, for the i-th agent $\x_i^{n+1} = \x_i(t^{n+1})$ is the position, and
$m_i^{n+1} = m_i(t^{n+1})$ the mass at time $t^{n+1}$. For the initialization of the algorithm, we set all agents to random positions $\{\x_i^0\}$
and all agents are given the uniformly distributed mass $\{m_i^0=\frac{1}{J}\}$. Then we proceed with all agents with $m_i^n>0$ in each iteration as follows:
\begin{align}
	\begin{dcases}
	m_i^{n+1} &=  m_i^n - \phi_p(\eta_i^n)m_i^n, \hspace{1cm} i\neq i_n \\
	m_{i_n}^{n+1} &= m_{i_n}^n + \sum\limits_{i\neq i_n} \phi_p(\eta_i^n)m_i^n \\
	&= 1- \sum\limits_{i\neq i_n} m_i^n + \sum\limits_{i\neq i_n} \phi_p(\eta_i^n)m_i^n \\
	m_+^{n+1}&:= \max\limits_{i\in J} m_i^{n+1} \\
	\x_i^{n+1}&= \x_i^n -h(\x_i^n,\lambda\psi_q(\tilde{m}_i^{n+1}))\nabla F(\x_i^n)
\end{dcases}
\label{eg:2.3}
\end{align}
with $i_n:=\argmin\limits_{i} F(\x_i^n)$ and $\tilde{m}_i^{n+1}=\frac{m_i^{n+1}}{m_+^{n+1}}$. First, we apply communication and redistribute the masses so that the best-positioned agent becomes the heaviest. After that, each agent is given a time step by the
backtracking method and we update the positions.
By computing $\x_+ := \argmax_{i\in J} F(\x_i^n)$ we find the \glqq worst \grqq agent, which will be eliminated. Repeating this will leave us with the heaviest agent.

\section{Example}
\label{section3}

To better understand how the SBGD method works, I would like to demonstrate an example. We are looking for the global minimum of the
function $F(x) = e^{\sin(2x^2)}+\frac{1}{10}(x-\frac{\pi}{2})$ on the interval $[-3,3]$.

\subsection{The swarms movement}

We apply the SBGD method with ten agents and the backtracking parameters $\lambda = 0.2$ and $\gamma = 0.9$.
For simplicity, we initialize the agents with equal distance to each other as shown in Figure \ref{fig:1}. Let the size of the points represent the
masses of the individual agents. Initially, all
agents have the same mass, therefore they are all the same size. The global minimum is located in the area $[1,2]$ and two agents are placed nearby. One of them is expected to approach the minimum. To the right in the subinterval $[2,3]$ is a local minimum, where another agent is closely placed. We expect this agent to become initially the heaviest agent of the swarm.\\

\begin{figure}[ht]
	\centering
\includegraphics[scale=0.33]{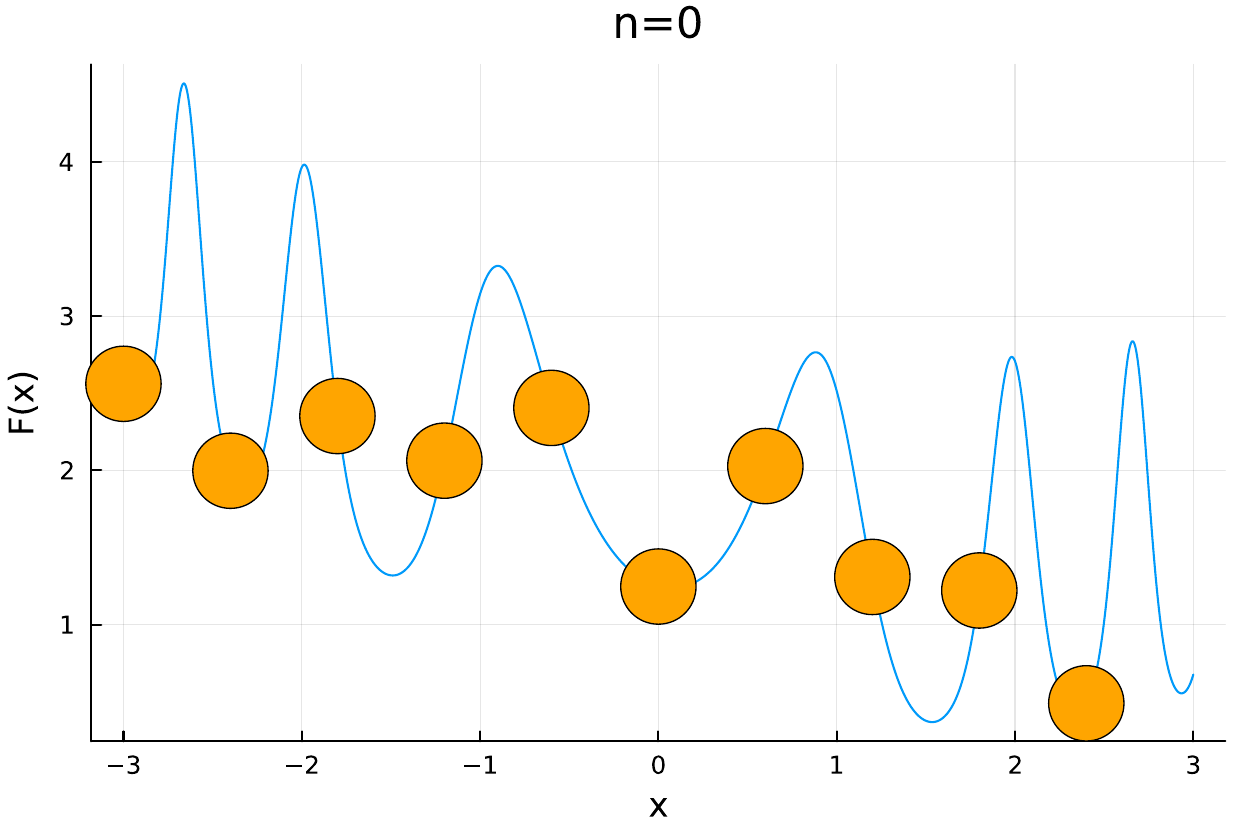}
\caption{
Agents (orange) on the graph of $F(x)$ in iteration $n=0$.
}
\label{fig:1}
\end{figure}

In Figure \ref{fig:2} we can observe the movement of the swarm. After one iteration the agents are already different in size and weight. As assumed, the agent near the local minimum pulls the mass of the others towards it. We also see that another agent is already approaching the global minimum. However, this one also loses its mass to the heaviest agent. After five iterations, two of the lighter agents have approached the global minimum. From now on, the mass distribution changes, as seen in iteration 7. Due to the better position near the global minimum, one of the lighter agents pulls mass from the heaviest agent and the remaining others to itself. Thus, it becomes the new heaviest agent and converges towards the global minimum in the further iterations.

\begin{figure}[ht]
	\centering
\includegraphics[scale=0.33]{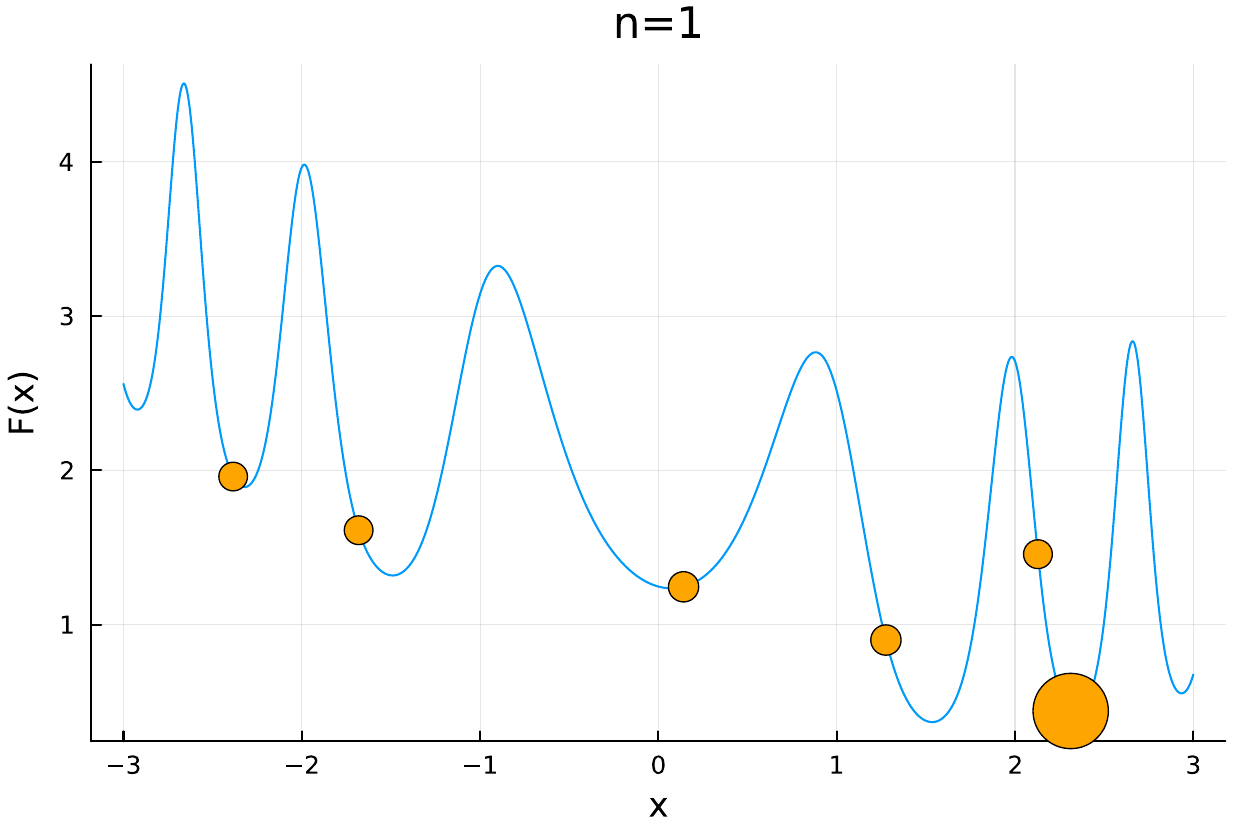}
\includegraphics[scale=0.33]{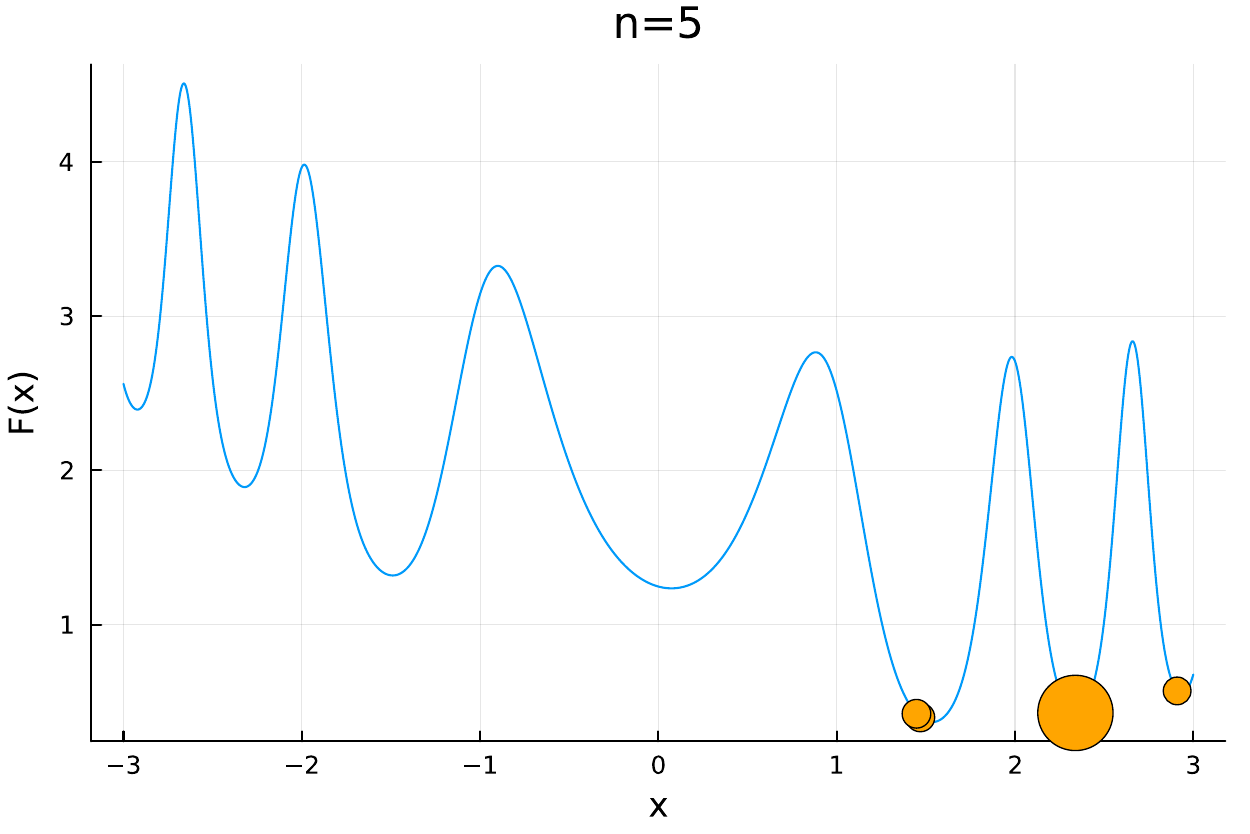}
\includegraphics[scale=0.33]{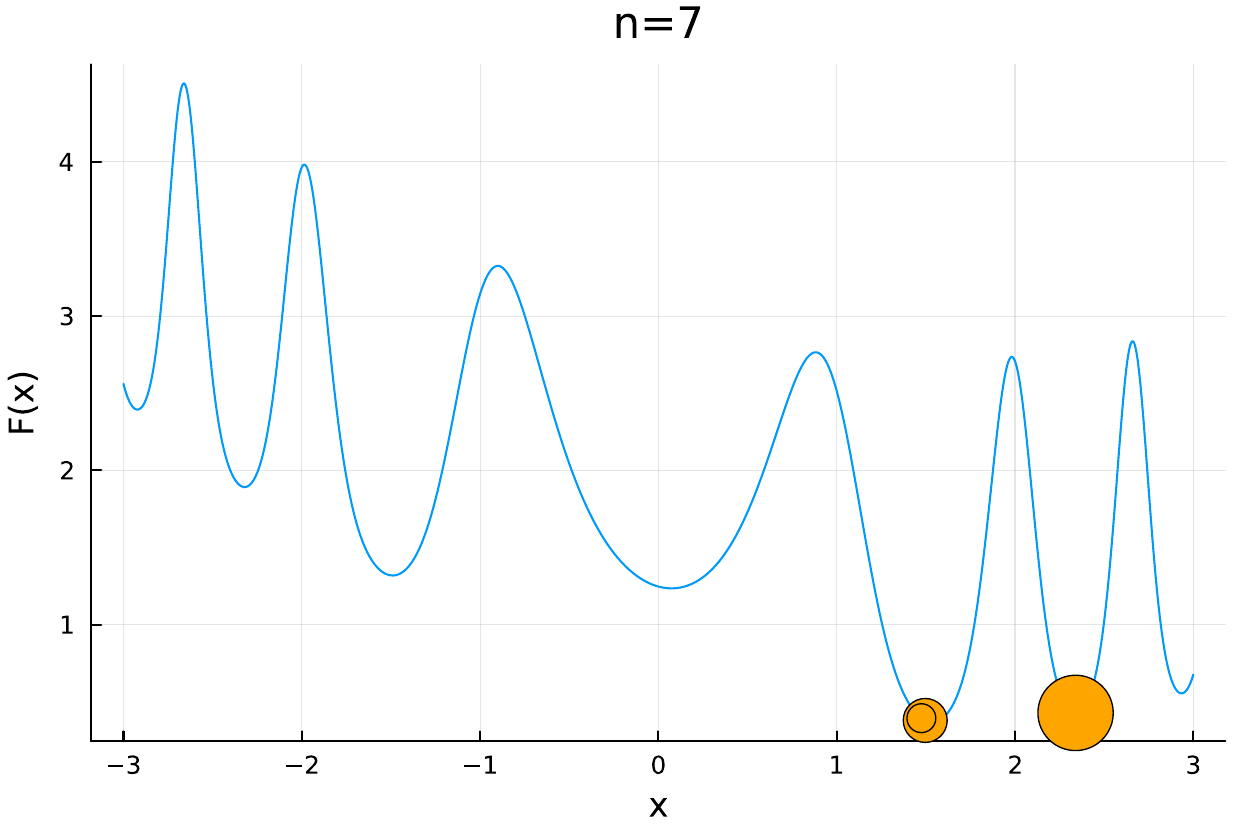}
\includegraphics[scale=0.33]{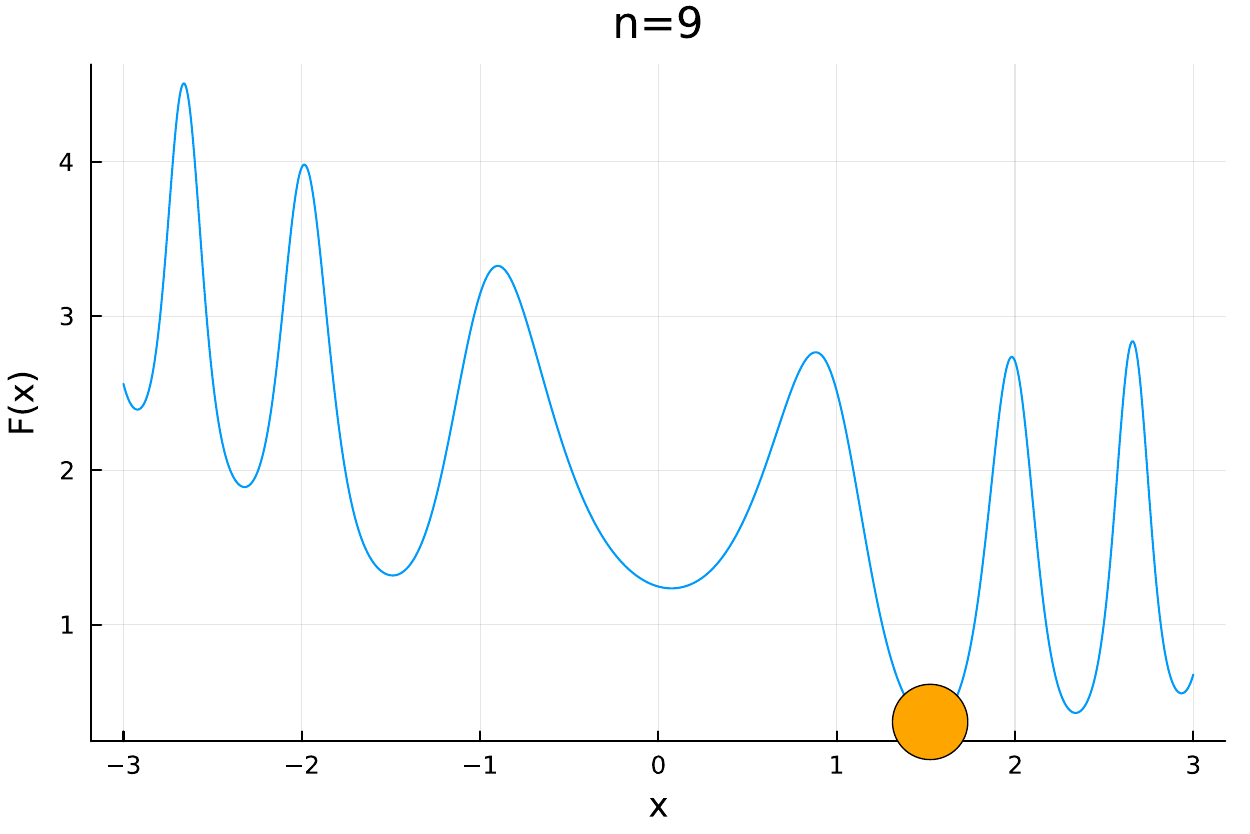}
\caption{
Agents (orange) on the graph of $F(x)$ in iteration $n=1,5,7,9$.
}
\label{fig:2}
\end{figure}

\clearpage
\subsection{Communication is the key}
\label{section 3.2}
The communication aspect is the key element in the efficiency of SBGD. In \cite{tad} the authors compared the SBGD method with backtracking gradient descent method
and the Adams method. Compared to both methods, SBGD had an overall better performance. In addition, I want to show the advantages of SBGD with a visual comparison to the backtracking gradient descent.
\\\\
Therefore consider the SBGD method but without communication between the agents. That means there is no mass transition, $m_i^n \equiv \frac{1}{J}$, and
$\psi_q(\tilde{m}_i^n)\equiv 1$ yields
\begin{align*}
	\x_i^{n+1}&= \x_i^n -h(\x_i^n,\lambda)\nabla F(\x_i^n), \hspace{1cm} i=1,\hdots, J.
\end{align*}
If we apply this method with the same backtracking parameters on the example before, we notice a different movement behavior of the agents towards the global minimum  (see figure \ref{fig:3}).
The agents flock together in groups and are trapped in the basins of local minima. Because of the equidistant initialization over the whole interval, one group of agents is able to reach the global minimum.
But the initial starting positions of the agents determine if they can reach the global minimum or not. As shown in figure \ref{fig:4}, if we move all agents
to the left side of the interval, the backtracking gradient descent stops before the global minimum. On the other hand, the SBGD method leads the swarm
further to reach the wanted global minimum.

\begin{figure}[ht]
	\centering
\includegraphics[scale=0.3]{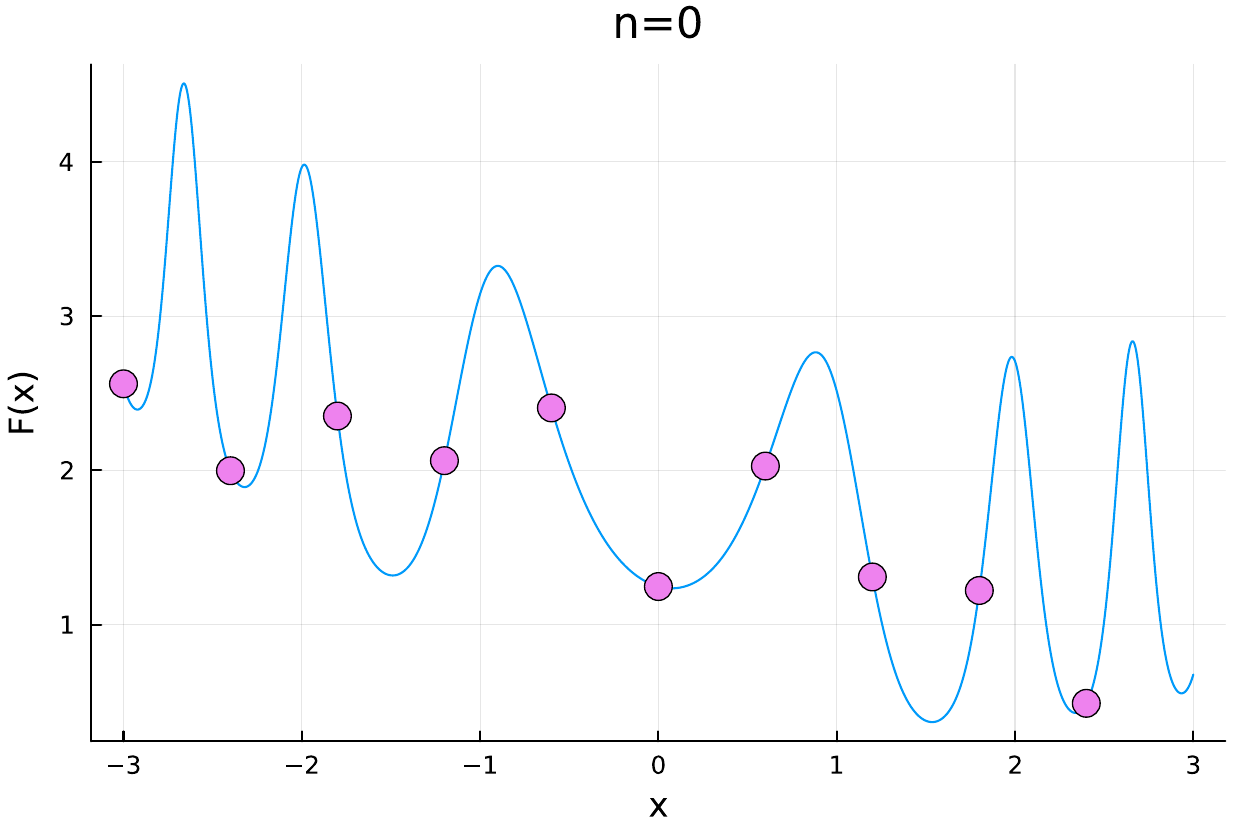}
\includegraphics[scale=0.3]{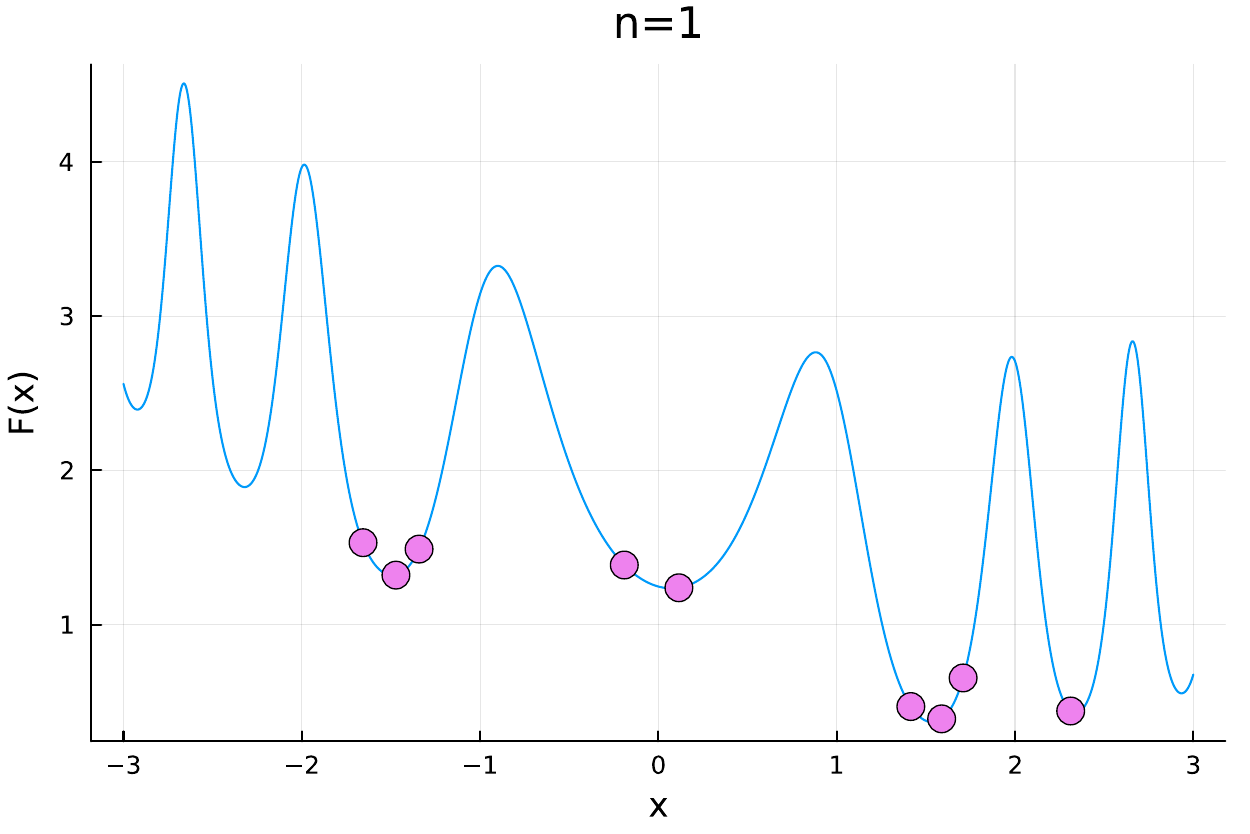}
\includegraphics[scale=0.3]{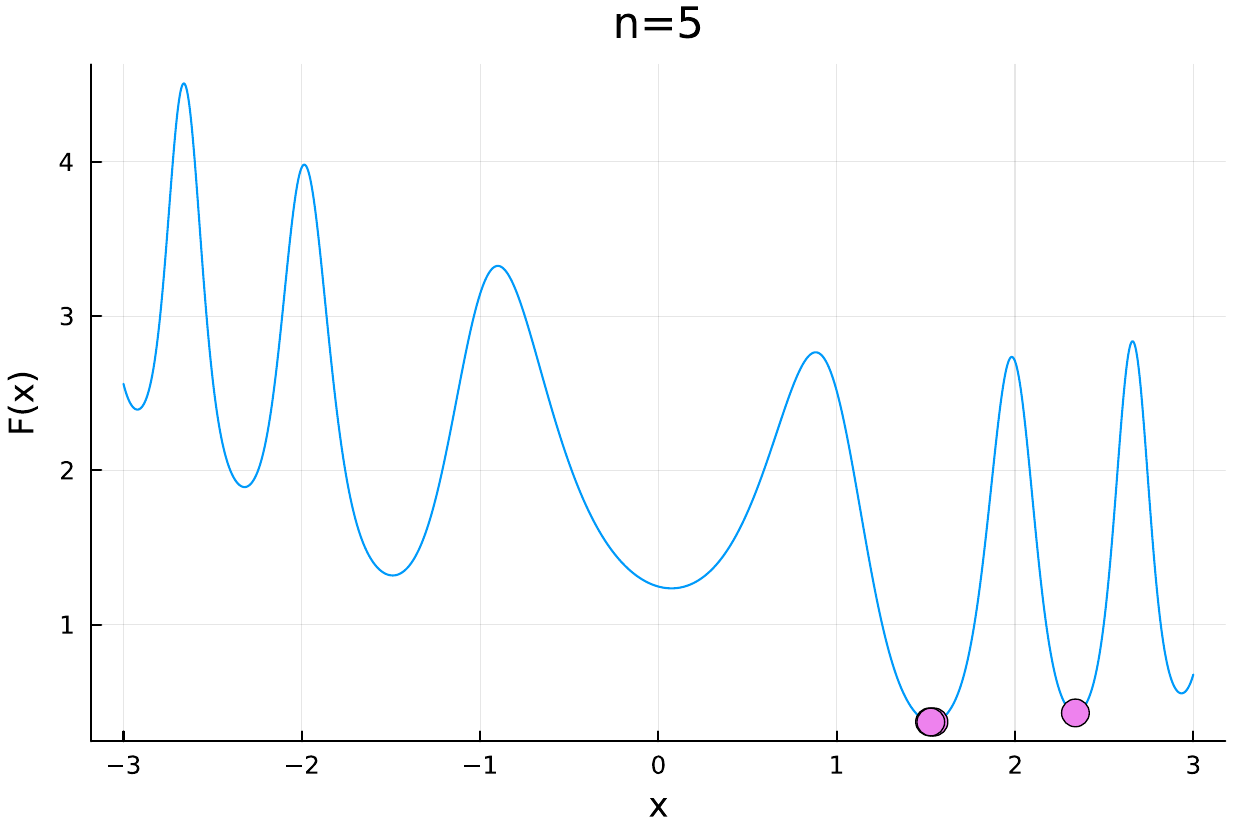}
\includegraphics[scale=0.3]{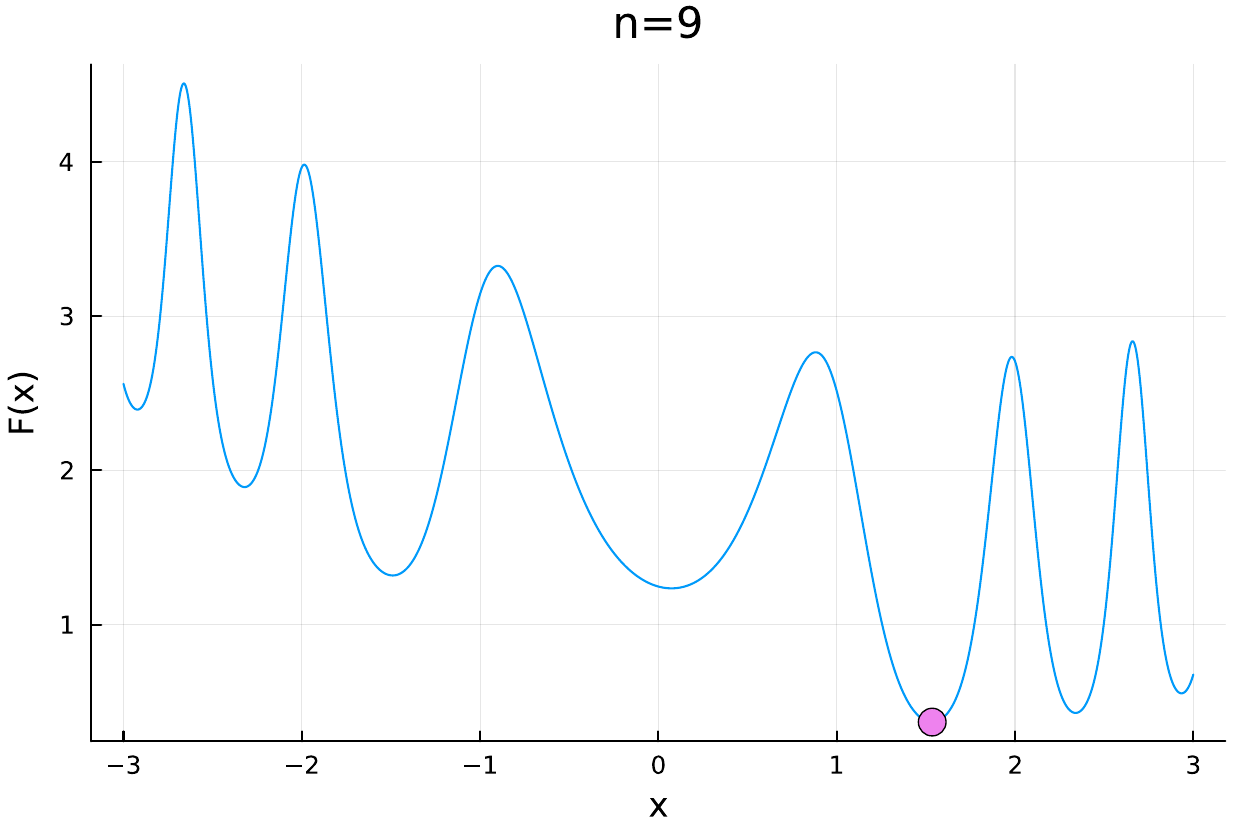}
\caption{
Agents (violet) progress with backtracking gradient descent.
}
\label{fig:3}
\end{figure}

\begin{figure}[ht]
	\centering
\includegraphics[scale=0.3]{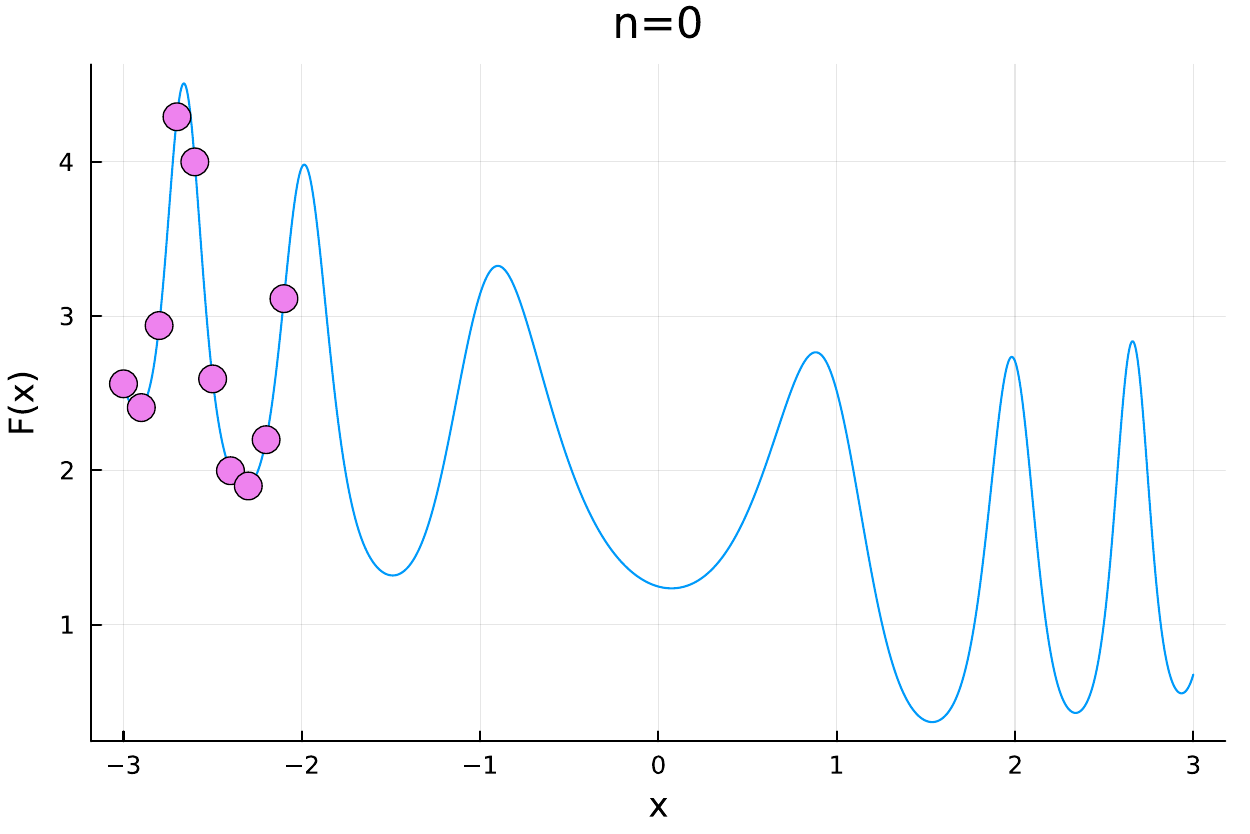}
\includegraphics[scale=0.3]{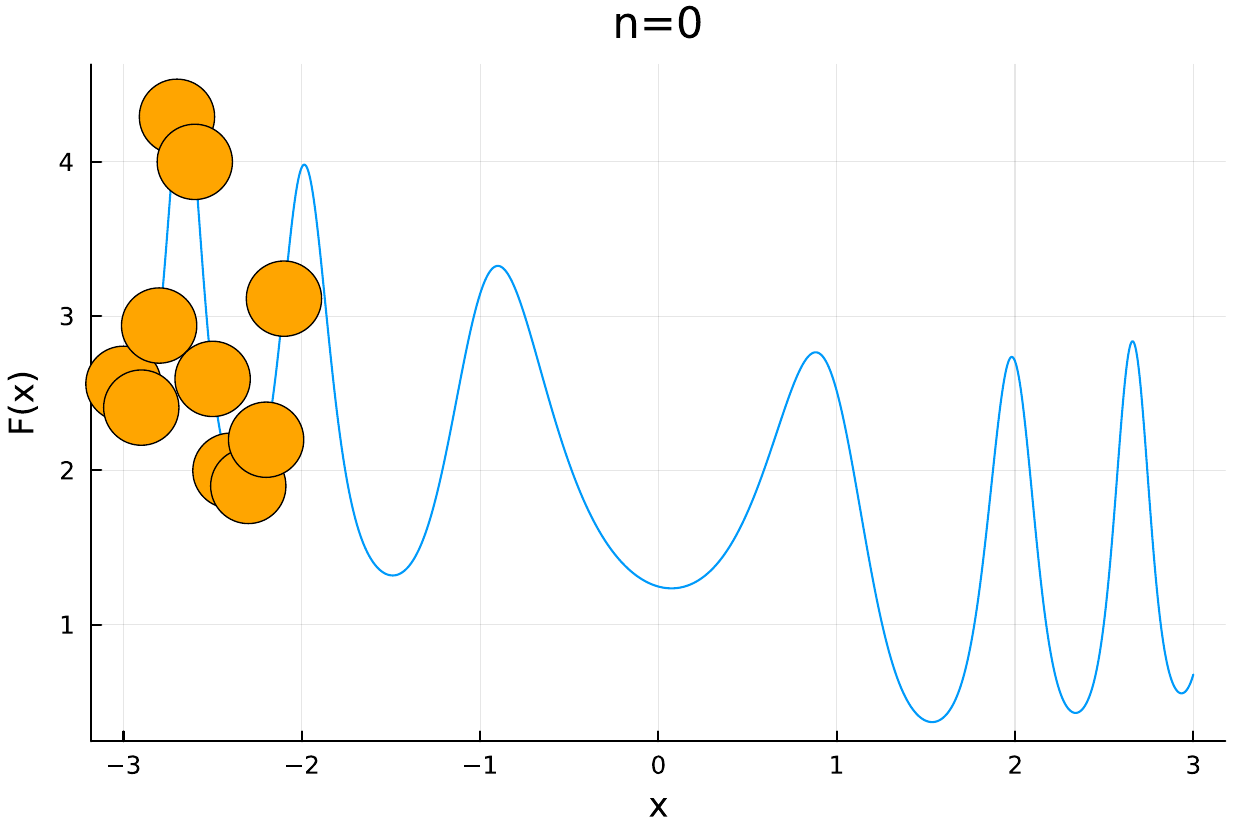}

\includegraphics[scale=0.3]{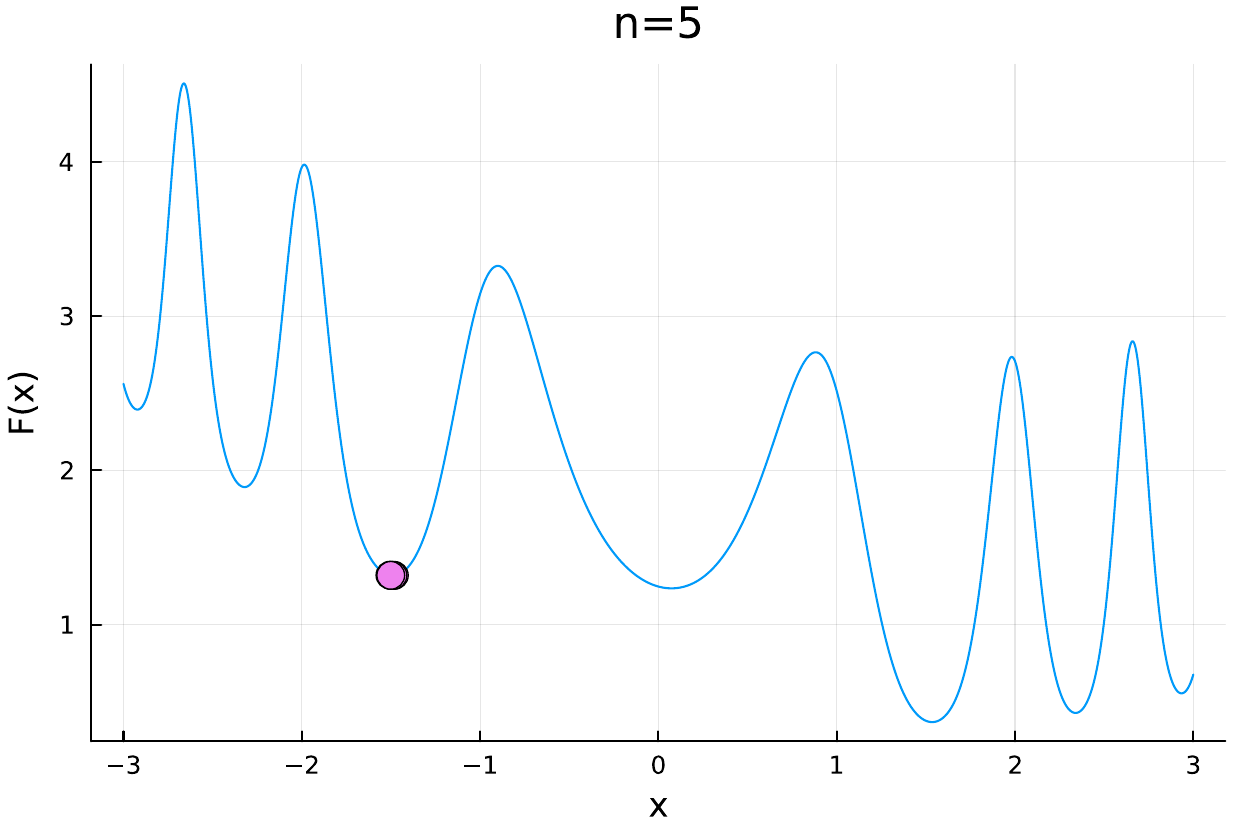}
\includegraphics[scale=0.3]{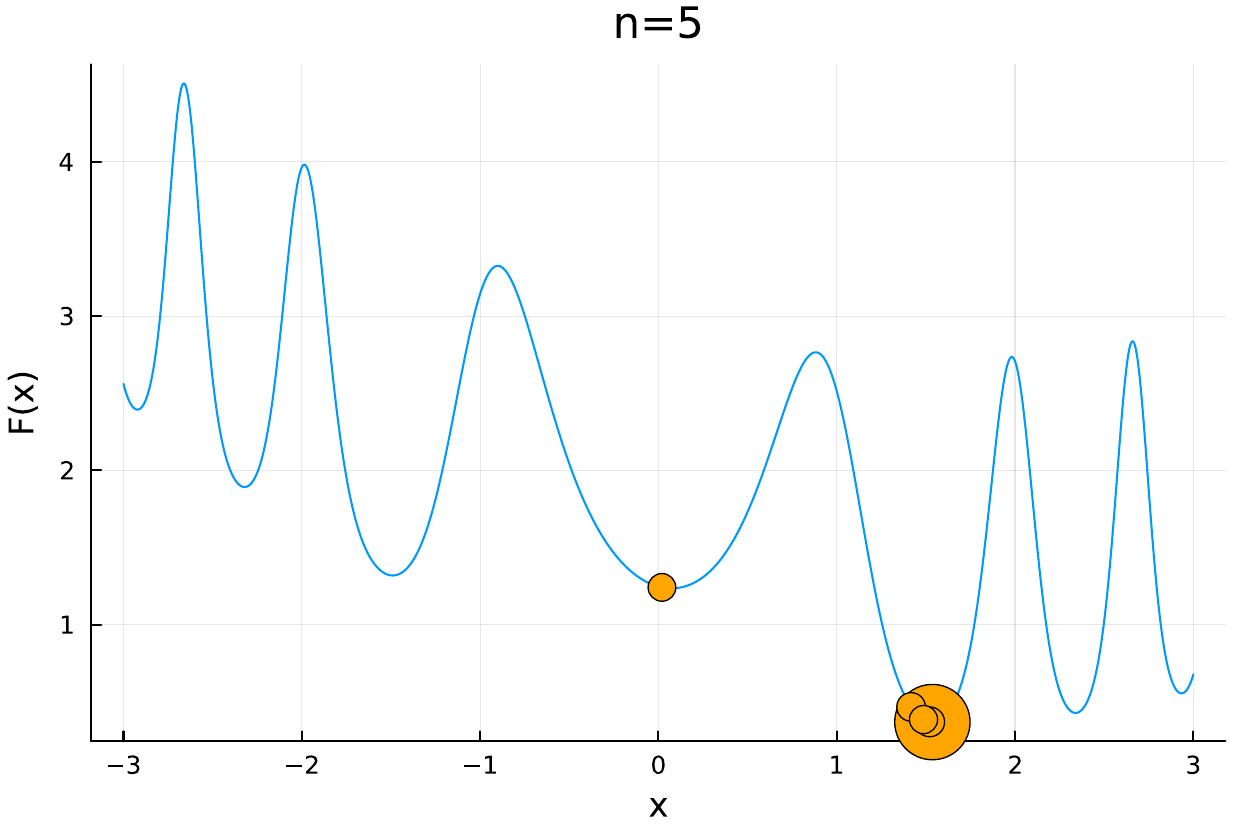}

\includegraphics[scale=0.3]{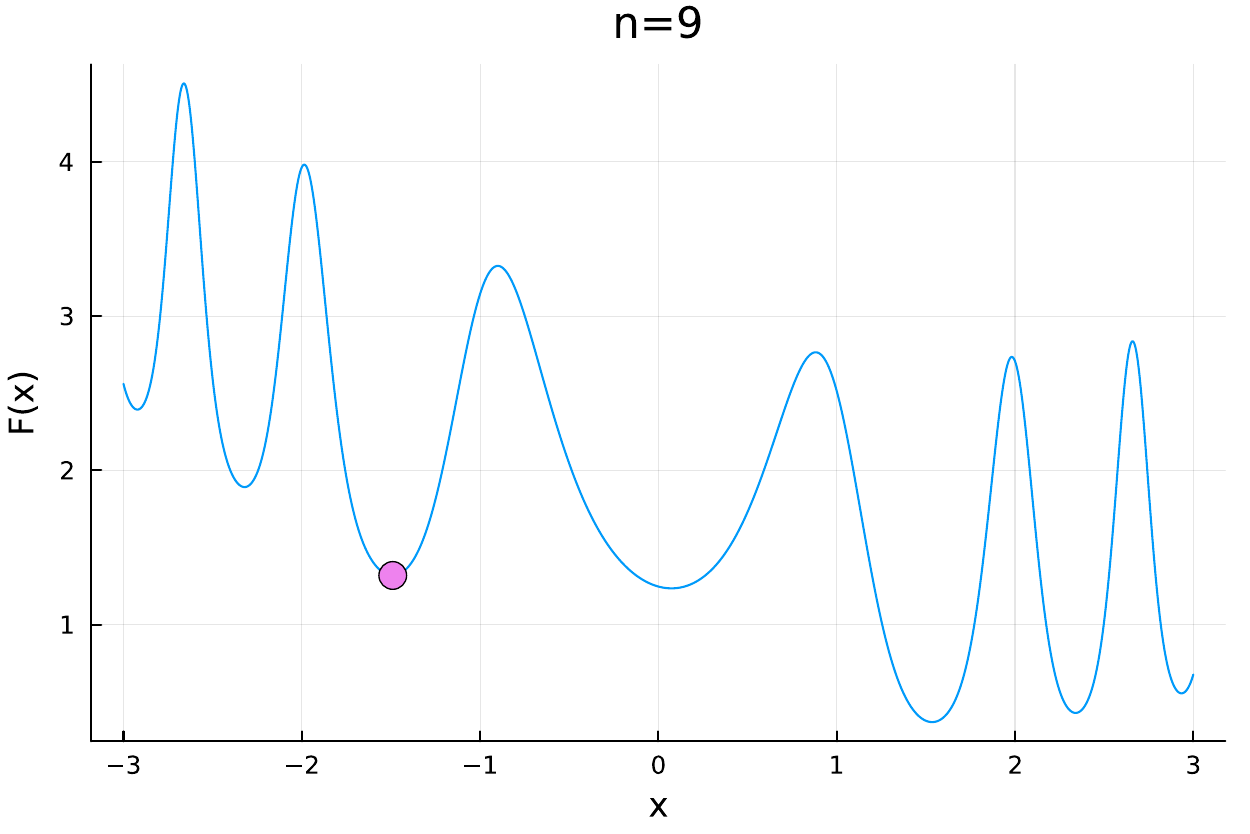}
\includegraphics[scale=0.3]{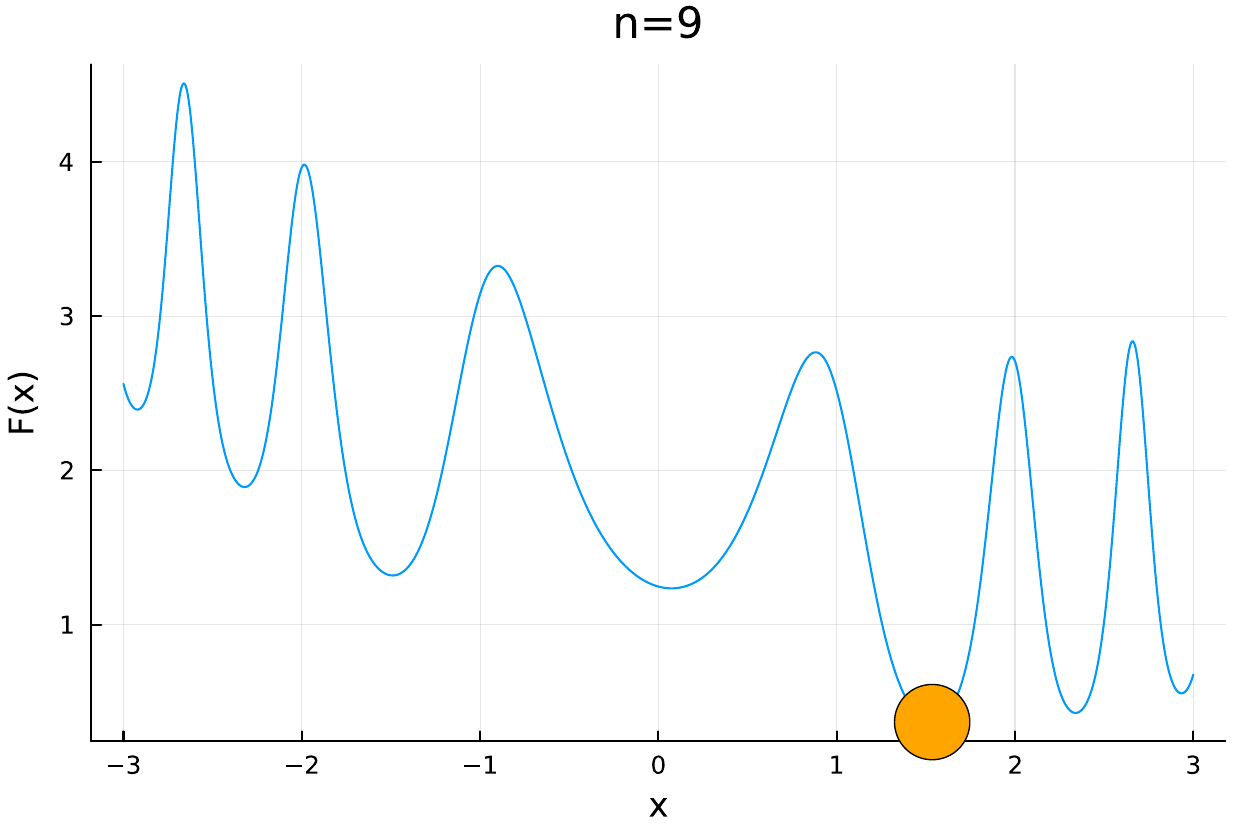}
\caption{
Agents progress with backtracking descent on the left and SBGD on the right.
}
\label{fig:4}
\end{figure}

\clearpage

\subsection{The parameters p and q}
By default we assume $p=q=1$. The question is, how does changing these two parameters affect the $SBGD_{p,q}$ method and its results?
In \cite{tad} is mentioned, that the parameter $p$ has a low influence level, while $q$ has more significant influence and therefore can be used for
fine-tuning purposes. As mentioned in \ref{section 2.3}, $q$ determines the influence of the relative mass. If $q$ is larger, $\psi_q(\tilde{m}_i^{n+1})$ will be
smaller and agents with a relative mass in the middle range receive larger time steps from the backtracking method. \\\\

To demonstrate the influence of $q$, which affects the swarms movement, I want to consider the two cases with different starting positions (equidistant vs left sided) from before.
For the case with equidistant starting positions as shown in figure \ref{fig:5}, we notice that more agents approach the global minimum with $SBGD_{1,2}$, than with $SBGD_{1,1}$.
Moreover, mass gaining of the minimizer happens faster with $SBGD_{1,2}$ due to its better position, gained from a better stepsize.
The worst case left sided scenario however (see figure \ref{fig:6}), seems to be more challenging for $SBGD_{1,2}$, than for $SBGD_{1,1}$. While the agents with $SBGD_{1,1}$ approach the global minimum
directly, the agents with $SBGD_{1,2}$ are more widely spread over the interval. Hence the minimizer gains weight later with $SBGD_{1,2}$, than with $SBGD_{1,1}$. \\\\
Therefore, we find that the speed and the movement of the swarm are influenced by $q$. Although the swarms behave different for both cases, in the end the global minimum is reached visually. But we need to evaluate, if both results are equally good.
In table \ref{tbl:1} we see the deviation from the result for the second case (left sided starting positions) with different
$p$ and $q$ pairings. First of all, we can agree with \cite{tad} that $p$ has no significant influence on the results. However, increasing $q$ has an effect.
But increasing $q$ alone does not lead to better results. We can see, that for SBGD with ten agents, an increasing of $q$ leads to worse results.
On the other hand, increasing both $q$ and the number of acting agents provide us with better results. For this case, $SBGD_{1,3}$ with 20 agents
has the best results. Although increasing agents seems to have more influence, than increasing $q$. For $q=1$ the results for 20 and 50 agents are equal to
$q=2$. Therefore we can conclude, for SBGD to provide us with good results, we need to consider different numbers of agents and different values of $q$.\\\\

\begin{table}[ht]
    \centering
    \begin{tabular}{c|c|c|c|c}
        \# Agents       & $p=1$            & $p=2$            & $p=3$ \\
        \hline \hline
        10              & $\e{5.32}{-5}$   & $\e{4.00}{-4}$   & $\e{3.90}{-3}$\\
        20              & $\e{1.16}{-6}$   & $\e{1.31}{-6}$   & $\e{1.06}{-6}$ & $q=1$\\
        50              & $\e{1.17}{-6}$   & $\e{1.17}{-6}$   & $\e{1.17}{-6}$\\
        \hline
        10              & $\e{1.60}{-3}$   & $\e{6.00}{-4}$   & $\e{2.03}{-5}$\\
        20              & $\e{2.10}{-6}$   & $\e{1.24}{-6}$   & $\e{1.16}{-6}$ &$q=2$\\
        50              & $\e{1.17}{-6}$   & $\e{1.17}{-6}$   & $\e{1.17}{-6}$\\
        \hline
        10              & $\e{1.20}{-2}$   & $\e{6.60}{-3}$   & $\e{1.23}{-2}$\\
        20              & $\e{9.88}{-8}$   & $\e{6.60}{-7}$   & $\e{6.20}{-7}$ &$q=3$\\
        50              & $\e{1.17}{-6}$   & $\e{1.17}{-6}$   & $\e{1.17}{-6}$\\
    \end{tabular}
    \caption{Deviation from $\x^* \approx 1.5355$ with $SBGD_{p,q}$ on the leftsided case.}
    \label{tbl:1}
\end{table}

\begin{figure}[ht]
	\centering
\includegraphics[scale=0.3]{plot2.pdf}
\includegraphics[scale=0.3]{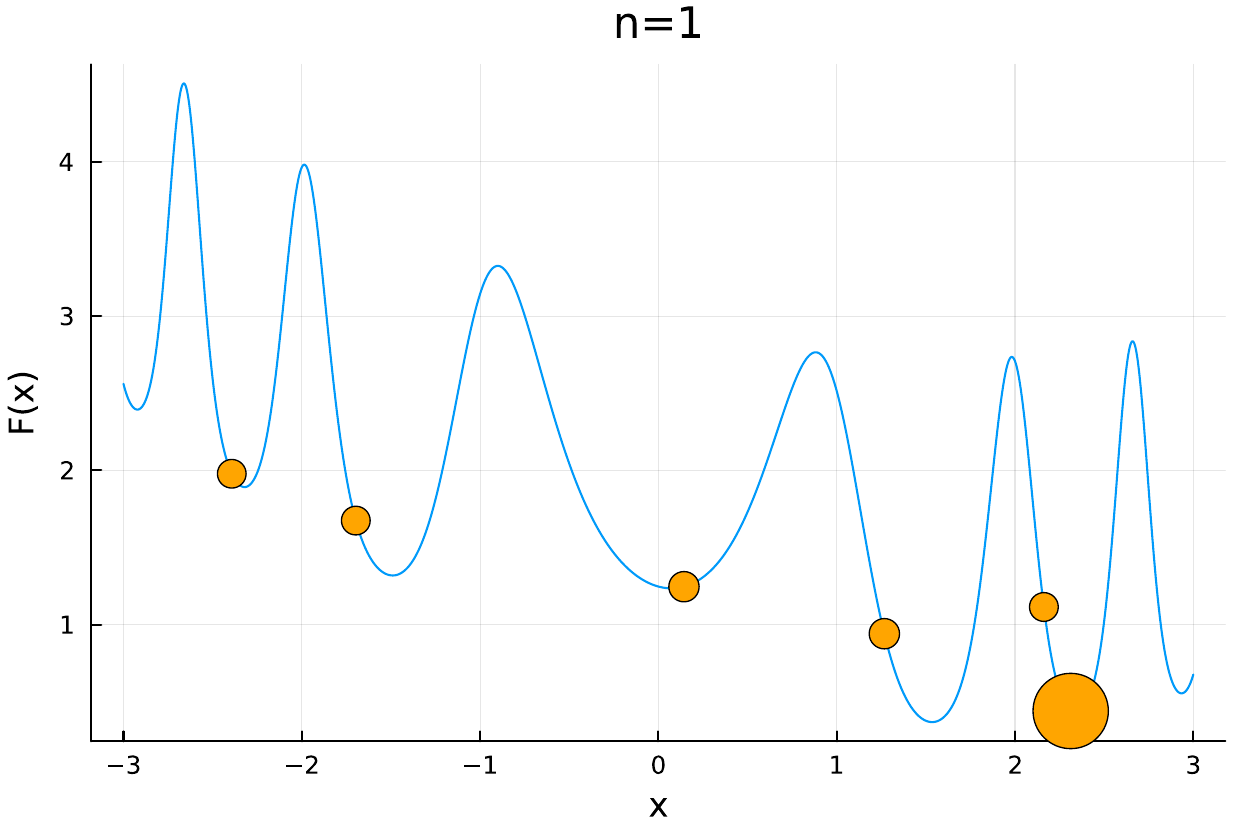}

\includegraphics[scale=0.3]{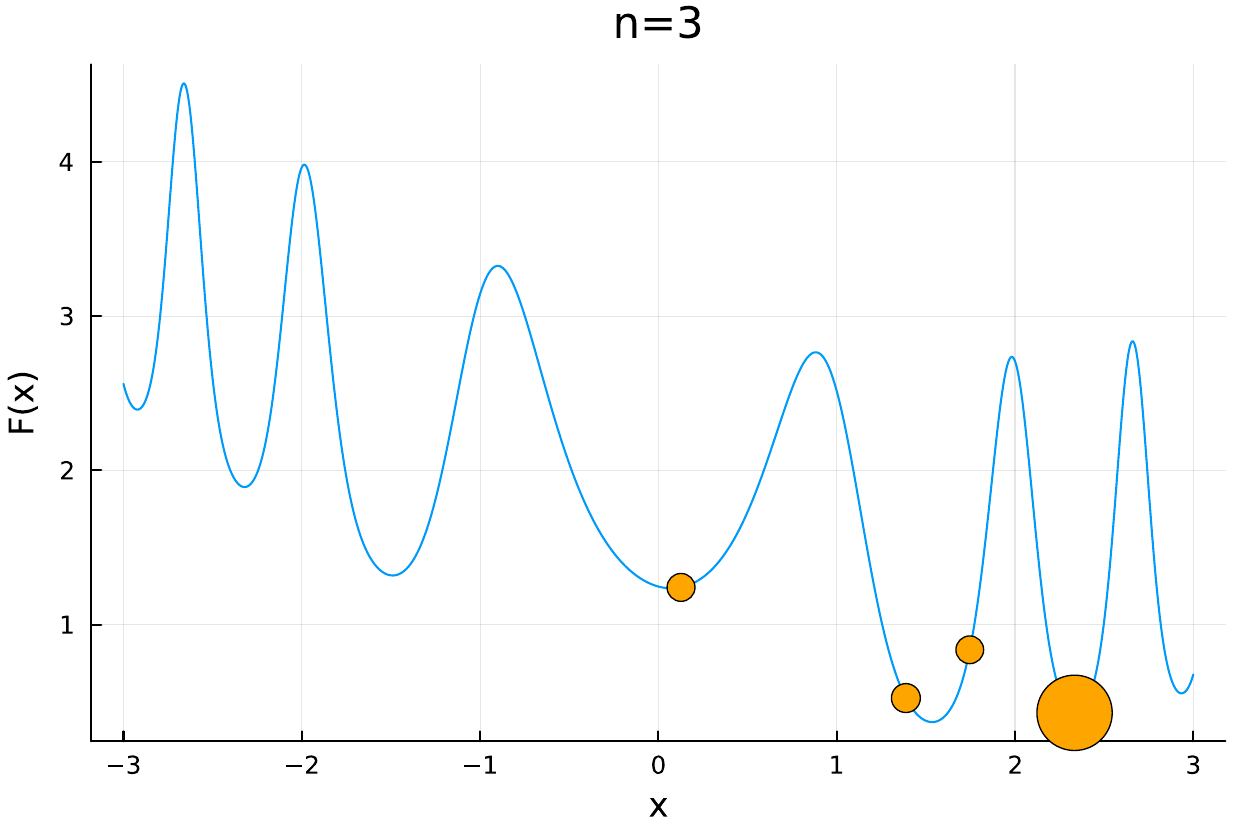}
\includegraphics[scale=0.3]{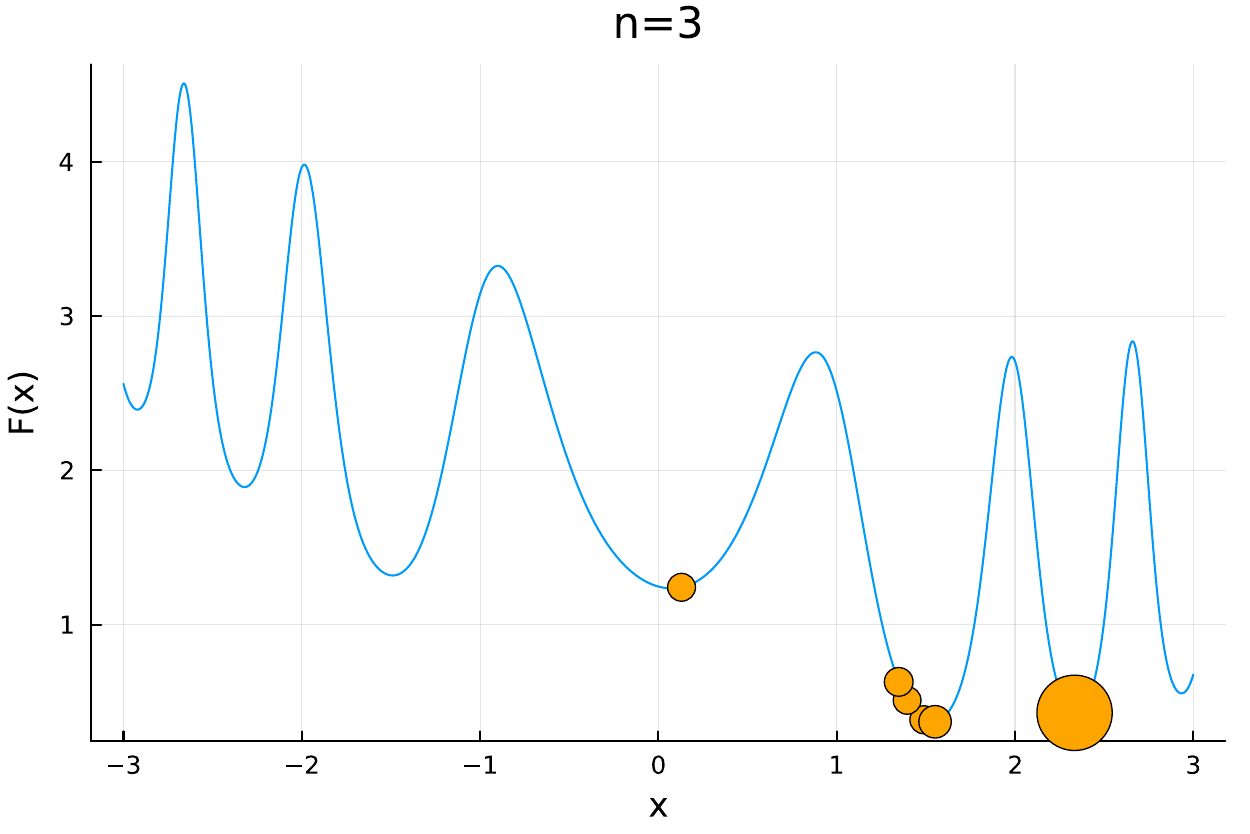}

\includegraphics[scale=0.3]{plot3.pdf}
\includegraphics[scale=0.3]{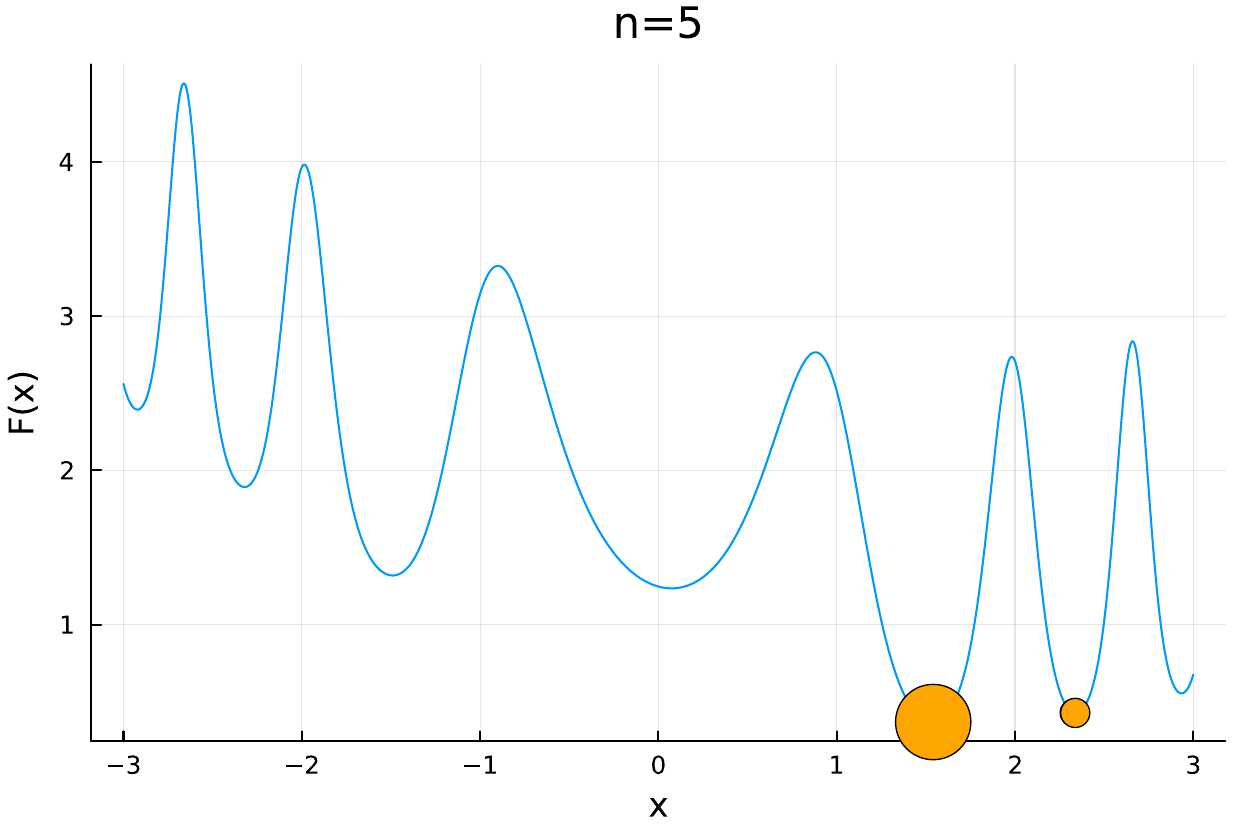}
\caption{
$SBGD_{1,1}$ on the left and $SBGD_{1,2}$ on the right with equidistant starting positions.
}
\label{fig:5}
\end{figure}

\begin{figure}[ht]
	\centering
\includegraphics[scale=0.3]{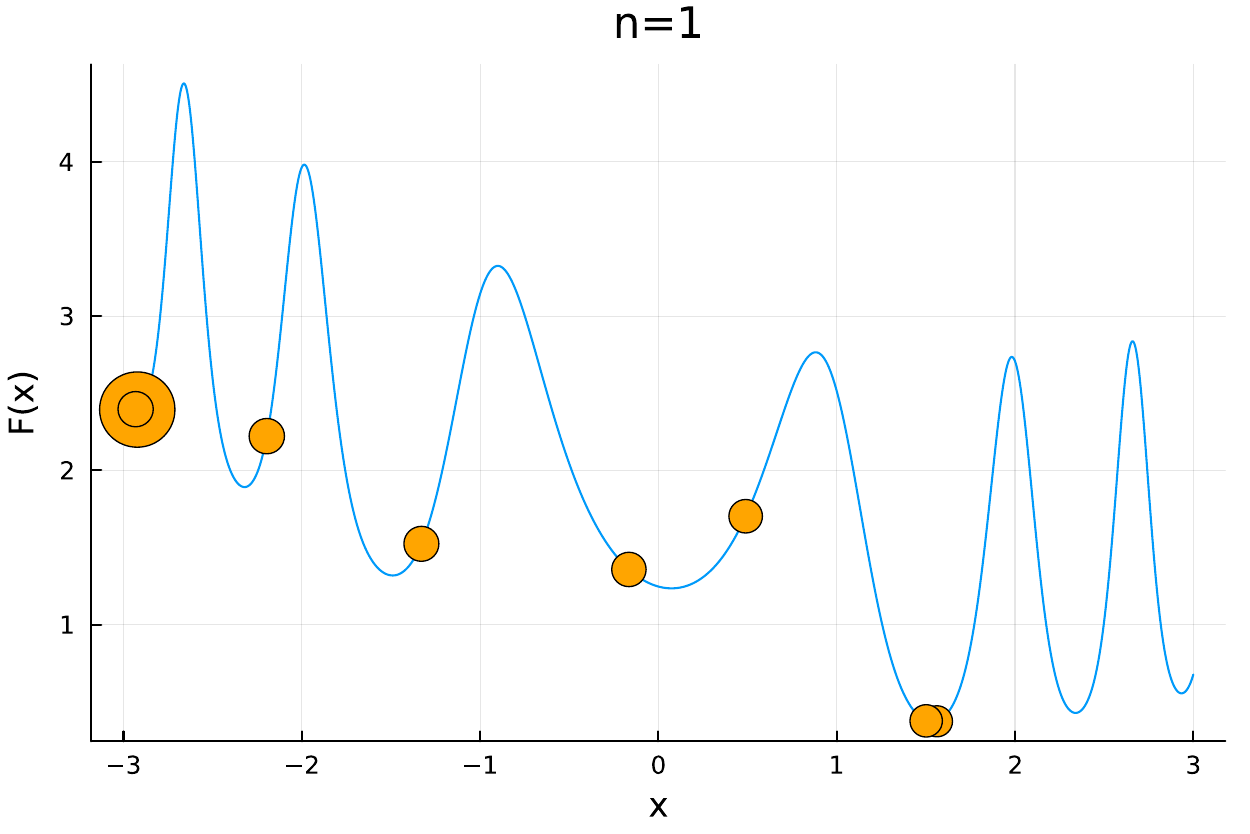}
\includegraphics[scale=0.3]{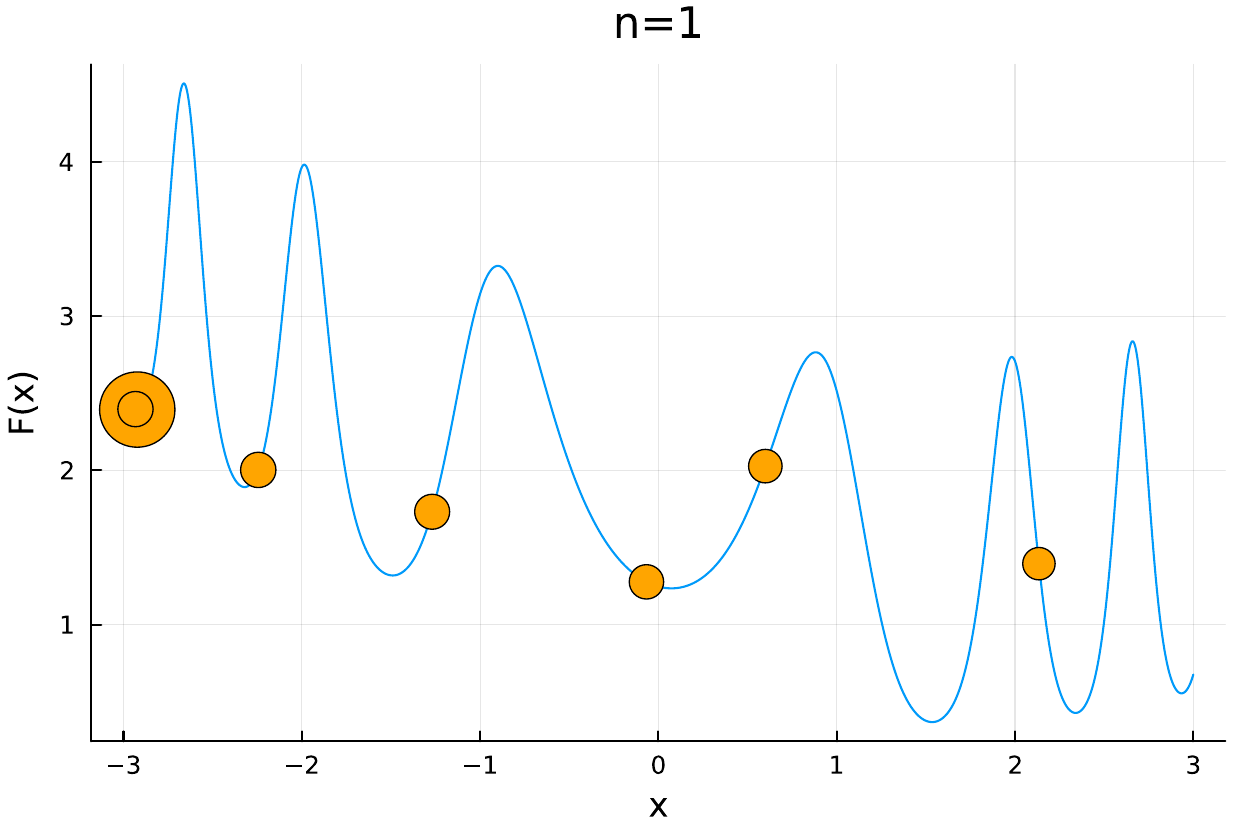}

\includegraphics[scale=0.3]{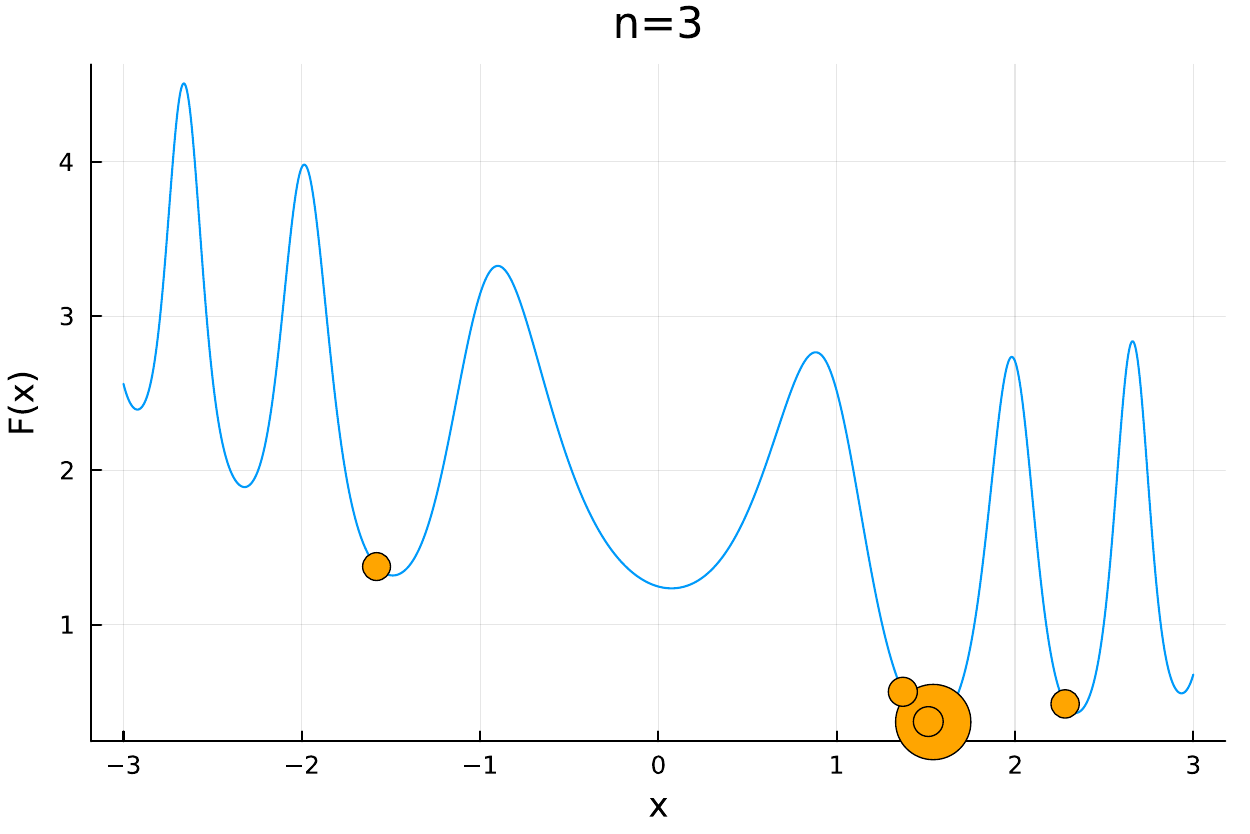}
\includegraphics[scale=0.3]{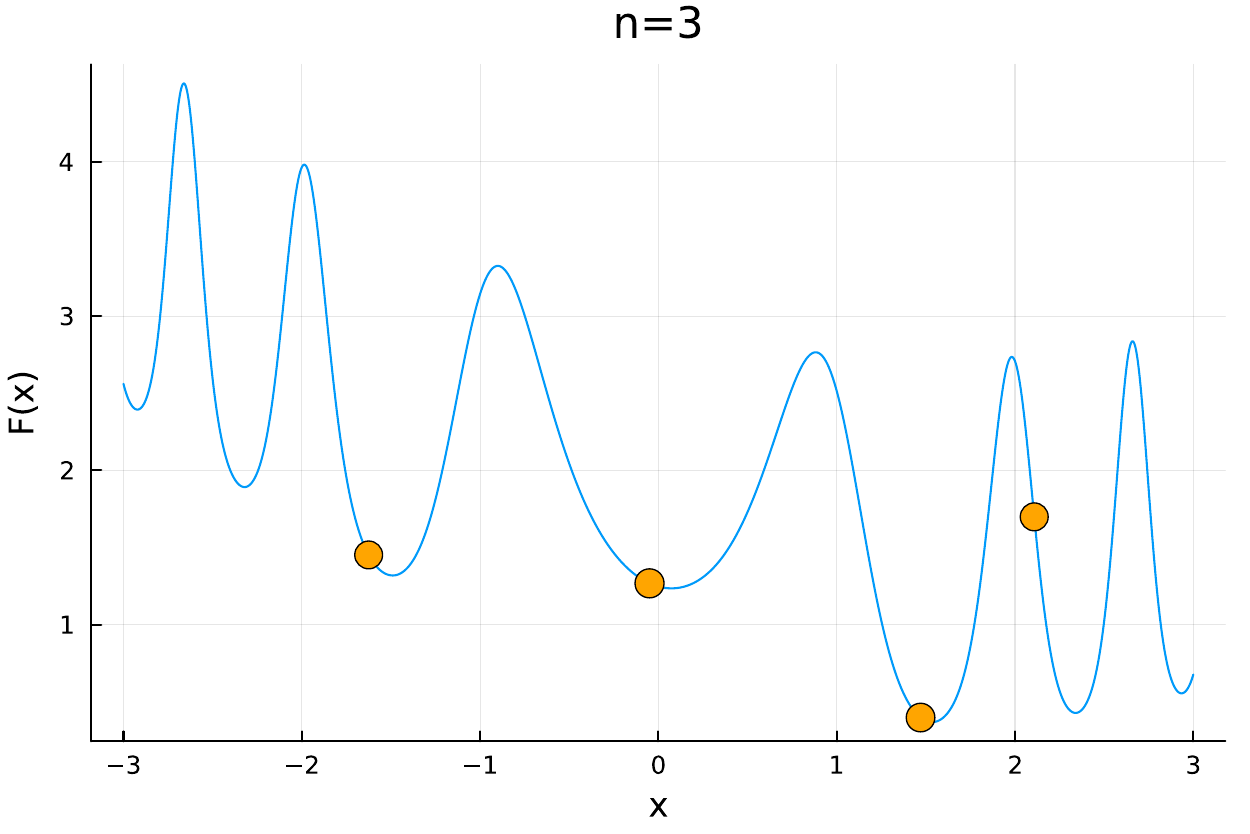}

\includegraphics[scale=0.3]{plot22.pdf}
\includegraphics[scale=0.3]{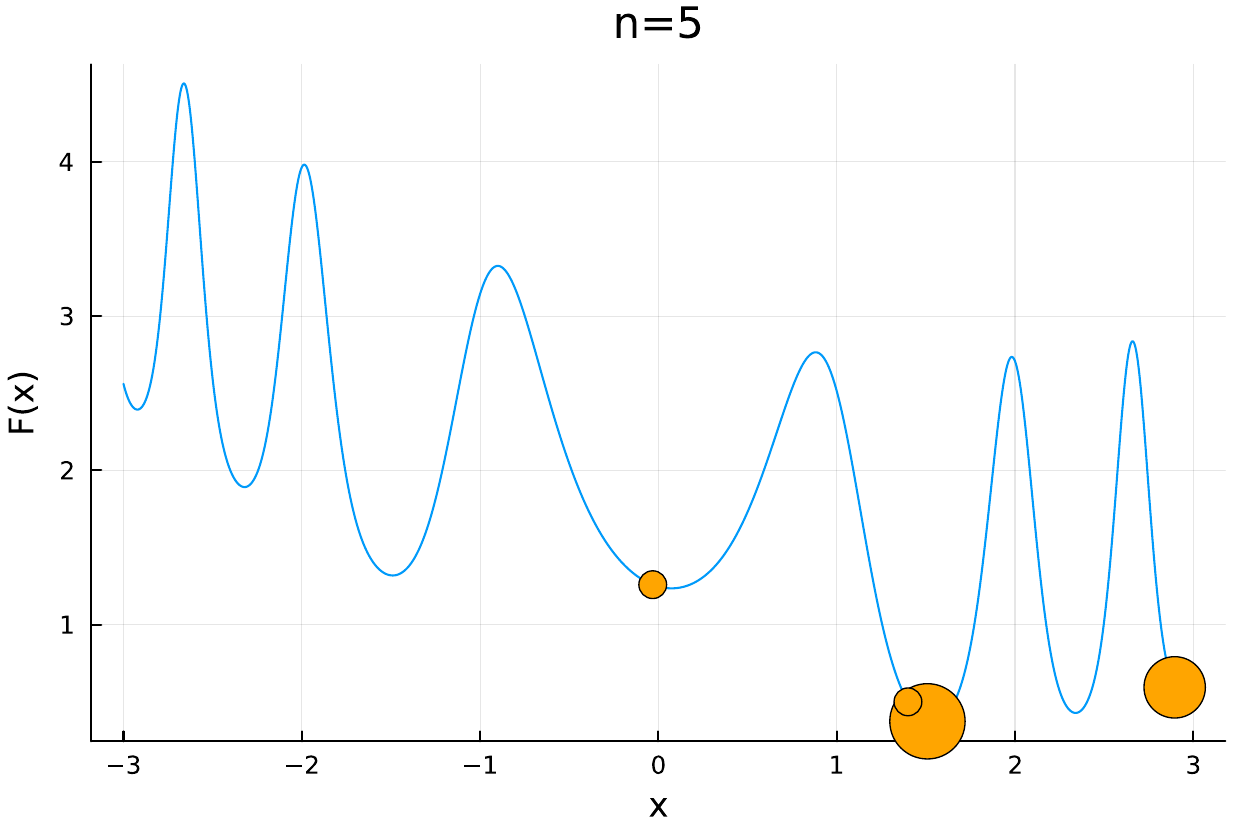}
\caption{
$SBGD_{1,1}$ on the left and $SBGD_{1,2}$ on the right with left sided starting positions.
}
\label{fig:6}
\end{figure}

\clearpage
\section{Implementation}
In the following we consider the  SBGD method for a one-dimensional function in Julia. The parameters will be set as default, thus $p=q=1$. First, we review the basic version of pseudocode as shown in algorithm \ref{alg:backtracking} and \ref{alg:SBGD}, before I introduce my implemented version.
Afterwards I discuss the usage of tolerance factors.

\subsection{Pseudocode}
\label{pseudocode}
We start with the \emph{backtracking line search}, since we use it for the time step protocol.
Therefore we set all parameters first and initialize a time step $h=h_0$ afterwards. In Lemma \ref{lemma1} is shown that we can use
\begin{align*}
	h_0 = \frac{2}{L}(1-\lambda\psi_q(\tilde{m}_i^{m+1})).
\end{align*}
Then  we proceed with a \emph{while}-loop to decrease the step length $h$ as shown in section \ref{section 2.3}. \\\\

\begin{algorithm}[H]
\begin{algorithmic}[1]
\State Set parameters $\lambda,\gamma\in(0,1)$
\State Set relative mass $\tilde{m}_i^{n+1} = \frac{m_i^{n+1}}{m_+}$
\State Initialize time step $h=h_0$
\While{$F(\x_i^{n}-h\nabla F(\x_i^n)) > F(\x_i^n)-\lambda h \vert \nabla F(\x_i^n)\vert ^2$}
\State $h \gets h\gamma$
\EndWhile
\State Set $h(\x_i^n,\lambda\psi_q(\tilde{m}_i^{n+1}))\gets h$
\end{algorithmic}
\caption{Backtracking Line Search}
\label{alg:backtracking}
\end{algorithm}

For the main algorithm the SBGD method, we again first set all parameters. Then the agents will be randomly placed and each is given the initial mass.
In each iteration the best placed agent receives masses from the other agents. To update the masses according to \eqref{eg:2.3}, we need to compute the relative heights first. After that we then find the maximum mass from all agents. This is used to compute the relative masses. These are needed to compute the step length
using the \emph{backtracking line search}. After all agents have updated positions, we eliminate the highest placed agent and reduce the number of active agents.

\begin{algorithm}[H]
\begin{algorithmic}[1]
\State Set $p,q>0$
\State Set $J = \text{Number of agents}$
\State Initialize random positions $\x_1^0,,\hdots, \x_J^0$ with a distribution $\rho_0$
\State Initialize masses $m_1^0,\hdots,m_J^0 = \frac{1}{J}$
\State Set best agent $i_0=\argmin\limits_{i=1,\hdots,J}F(\x_i^0)$
\While{$J>2$}
\State Set $F_-^n=F(\x_{i_n}^n),F_+^n=\max\limits_{i=1,\hdots,J}F(\x_i^n)$
\For{$i=1,\hdots,J$ und $i\neq i_n$}
\If{$m_i^n > 0$}
    \State Compute realtive height $\eta_i^n=\frac{F(\x_i^n)-F_{\min}^n}{F_{\max}^n-F_{\min}^n}$
    \State $m_i^{n+1}=m_i^n-\phi_p(\eta_i^n)m_i^n$
\EndIf
\EndFor
\State $m_{i_n}^{n+1} = m_{i_n}^n + \sum\limits_{i\neq i_n} \phi_p(\eta_i^n)m_i^n$
\State Compute $m_+=\max\limits_{i=1,\hdots,J} m_i^{n+1}$
\For{$i=1,\hdots,J$}
\State Compute relative mass $\tilde{m}_i^{n+1} = \frac{m_i^{n+1}}{m_+}$
\State Compute step length using backtracking $h=h(\x_i^n,\lambda\psi_q(\tilde{m}_i^{n+1}))$
\State $\x_i^{n+1}= \x_i^n -h\nabla F(\x_i^n)$
\EndFor
\State Eliminate worst agent $i_+ = \argmax\limits_{i=1,\hdots,J}F(\x_i^n)$
\State $J \gets J-1$
\EndWhile
\end{algorithmic}
\caption{Swarm-based gradient descent}
\label{alg:SBGD}
\end{algorithm}

\subsection{Implementation in Julia}
For my implementation I used the Julia Version 1.7.1. and I required only the package \emph{Random} to place the agents randomly.
To have a better overview, I divided the program in several functions, which I will explain below. \\

The main function \emph{SBGD}() (see listing \ref{lst:SBGD}) uses help functions to compute the global minimum with the SBGD algorithm.
The first step is to generate the agents. Therefore I created the function \emph{generateAgents}() (listing \ref{lst:generateAgents}). This function
needs the parameters $a,b$ as interval borders and $J$ for the number of agents. The agents are generated as $J\times 2$ array. The first column contains
the positions of the agents and the second column contains the masses.\\

\begin{lstlisting}[caption={Help function \emph{generateAgents}().},captionpos=b,label={lst:generateAgents}]
function generateAgents(J,a,b)
		agents = zeros(J,2)
		agents[1:J,2] .= 1/J
		positions = rand!(zeros(J),a:0.1:b)
		agents[1:J,1] = positions
		return agents
end
\end{lstlisting}

Next, I set the counter for active agents $N$ and $index$, which will hold the index of the heaviest agent. After that, I start the iterations with a \emph{while}-loop. In each iteration I begin with computing the best and worst placed agent. Therefore I am using two help functions \emph{searchMax}() (listing \ref{lst:searchMax}) and analogous to this \emph{searchMin}(). Julia already has functions for searching the maximum and minimum value from an array, but we can not use them because
we need to find the maximum and minimum from the still active agents. Hence, in \emph{searchMax}() and \emph{searchMin}() I check the condition if the mass is
not equal to zero. \\

\begin{lstlisting}[caption={Help function \emph{searchMax}().},captionpos=b,label={lst:searchMax}]
function searchMax(agents,J,F)
    max = 0
    index = 0
    for i=1:J
        if(agents[i,2]!=0)
            max = F(agents[i,1])
            index = i
            break
        end
    end
    for i=index+1:J
        if(agents[i,2]!=0 && F(agents[i,1])>max)
            max = F(agents[i,1])
            index = i
        end
    end
    return index
end
\end{lstlisting}

After that I use the obtained maximum and minimum position of the swarm to calculate the relative heights, which are used to compute the mass transitions.
The variable $sum$ holds the shedded mass from the agents, which the best placed agent receives after completing the \emph{for}-loop. \\\\
Before continuing with the updating of the positions, I eliminate the worst placed agent. After that I use the $findmax()$ function, which Julia provides,
to find the heaviest agent. This agents mass is used to calculate the relative masses, which are passed to the \emph{backtracking}() function (listing \ref{lst:backtracking}).
This function works as in algorithm \ref{pseudocode} already explained.\\\\

\begin{lstlisting}[caption={Help function \emph{backtracking}().},captionpos=b,label={lst:backtracking}]
function backtracking(lambda,gamma,q,mTilde,F,nablaF,x,L)
    mass = psi(q,mTilde)
    h = 2/L *(1-lambda*mass)
    while(F(x-h*nablaF(x))>F(x)-lambda*mass*h*abs(nablaF(x))^2)
        h = gamma * h
    end
    return h
end
\end{lstlisting}

This process is repeated until one agent remains, of which the position will be returned.
\newpage

\begin{lstlisting}[caption={Main function \emph{SBGD}().},captionpos=b,label={lst:SBGD}]
function SBGD(p,q,F,nablaF,J,a,b,lambda,gamma,L)
    agents = generateAgents(J,a,b)
    N = J
    index = 0
    while(N >= 2)
        max_index = searchMax(agents,J,F)
        fMax = F(agents[max_index,1])
        min_index = searchMin(agents,J,F)
        fMin = F(agents[min_index,1])

        sum = 0
        for i = 1:J
            if(i != min_index && agents[i,2]!=0)
                eta = (F(agents[i,1]) - fMin)/(fMax-fMin)
                change = phi(p,eta)*agents[i,2]
                sum += change
                agents[i,2] = agents[i,2]-change
            end
        end
        agents[min_index,2] = agents[min_index,2]+sum
        agents[max_index,2] = 0
        N = N-1

        maxMass = findmax(agents[1:J,2])[1]
        index = findmax(agents[1:J,2])[2]
        for i=1:J
            if(agents[i,2]>0)
                mTilde = agents[i,2]/maxMass
                h = backtracking(lambda,gamma,q,mTilde,F,nablaF,agents[i,1],L)
                agents[i,1] = agents[i,1] - h*nablaF(agents[i,1])
            end
        end
    end

    return agents[index,1]
end
\end{lstlisting}

\subsection{Usage of tolerance factors}

In section 3 of \cite{tad} three tolerance factors are introduced: \emph{tolm}, \emph{tolmerge} and \emph{tolres}.
These are used for a more optimized version of implementation. Because in the basic version shown before, one agent at a time will be eliminated, this
could take a while if we use a large number of agents. To avoid this, the usage of the tolerance factors is recommended.
\\\\
First, instead of eliminating one agent in each iteration, we eliminate all agents, whose masses are below our \emph{tolm} value. Because these agents are such lightly
weighted, we do not expect them to improve the global swarm position. \\
The second tolerance factor \emph{tolmerge} is used to determine if two agents are too close to each other. If they are too close to each other, instead of continuing with both, we merge them. This is because, even if we continue with both, only one might go further and gain mass in the next iterations and the other one will be eliminated at some point. Therefore to reduce computation time, we merge them. \\
Lastly, we use \emph{tolres} to determine, if we can stop the computation already. Instead of waiting until only one agent remains, we can
compute the residual between the best agent of the current and the last iteration. If the difference is small enough, we can stop the computation.

\begin{algorithm}[H]
\begin{algorithmic}[1]
\State Set $p,q>0$
\State Set $J = \text{Number of agents}$
\State Initialize random positions $\x_1^0,,\hdots, \x_J^0$ with a distribution $\rho_0$
\State Initialize masses $m_1^0,\hdots,m_J^0 = \frac{1}{J}$
\State Set best agent $i_0=\argmin\limits_{i=1,\hdots,J}F(\x_i^0)$
\While{$J>2$}
\State Set $F_-^n=F(\x_{i_n}^n),F_+^n=\max\limits_{i=1,\hdots,J}F(\x_i^n)$
\For{$i=1,\hdots,J$ und $i\neq i_n$}
\If{$m_i^n < \frac{1}{N}*\emph{tolm}$}
\State Set $m_i^{n+1}=0$
\State $J \gets J-1$
\Else
    \State Compute realtive height $\eta_i^n=\frac{F(\x_i^n)-F_{\min}^n}{F_{\max}^n-F_{\min}^n}$
    \State $m_i^{n+1}=m_i^n-\phi_p(\eta_i^n)m_i^n$
\EndIf
\EndFor
\State $m_{i_n}^{n+1} = m_{i_n}^n + \sum\limits_{i\neq i_n} \phi_p(\eta_i^n)m_i^n$
\State Compute $m_+=\max\limits_{i=1,\hdots,J} m_i^{n+1}$
\For{$i=1,\hdots,J$}
\State Compute relative mass $\tilde{m}_i^{n+1} = \frac{m_i^{n+1}}{m_+}$
\State Compute step length using backtracking $h=h(\x_i^n,\lambda\psi_q(\tilde{m}_i^{n+1}))$
\State $\x_i^{n+1}= \x_i^n -h\nabla F(\x_i^n)$
\EndFor
\State Merge the agents if their distance < \emph{tolmerge}
\State Set new best agent $i_{n+1}=\argmin\limits_{i=1,\hdots,J}F(\x_i^{n+1})$
\State Compute residual $res = \lvert \x_{i_{n+1}}^{n+1}-\x_{i_n}^n\rvert$
\If{$res<tolres$}
\State $\x_{SOL} \gets \x_{i_{n+1}}^{n+1}$
\State break
\EndIf
\EndWhile
\end{algorithmic}
\caption{Swarm-based gradient descent with tolerance factors}
\label{alg:tol}
\end{algorithm}

For the example from section \ref{section3} I used the thresholds mentioned in \cite{tad}. These are
\begin{align}
    tolm = 10^{-4}, \hspace{1cm} tolmerge = 10^{-3},\hspace{1cm} tolres = 10^{-4}.
\end{align}
In table \ref{tbl:tabelle} is shown, how the SBGD method performed with different numbers of agents.
The first case is, when all agents are equidistant distributed over the whole interval. We notice that the number of iterations is significantly smaller, than
the number of agents. Moreover, the more agents we use, the more precise the solution becomes. However the difference between 20, 50 and 100 agents is relatively
small. For 1000 agents the SBGD method astonishingly only needs one iteration and returns the best solution. Compared to that, it might not be useful to use
the tolerance factors for less than 20 agents. As seen in section \ref{section3} the agent, which converges to the global minimum in the end, first loses a lot of its mass to the heaviest agent, before it gains it back. Therefore the factor \emph{tolm} might lead to an early elimination of the minimizer. \\
If we compare all this to the worst case scenario, where all agents are initialized on the left end of the interval, we notice a slightly worse behavior of the SBGD method. Although there are still many fewer iterations used, the results are different. The tolerance factors seem to not worsen the case with 10 agents, like before. On the contrary, it seems to be slightly better using 10 agents than 50 or 100.

\begin{table}
    \centering
    \begin{tabular}{c|c|c}
        \# Agents  & $\vert \x^*-\x_{SOL}\vert$ & \# Iterations \\
        \hline\hline
        10                         & 0.7768             & 1                             \\
        20                         & 0.0021             & 4                             \\
        50                         & 0.0026             & 2                             \\
        100                        & 0.0014             & 2                             \\
        1000                       & 0.0002             & 1                             \\
    \end{tabular}
    \quad
    \begin{tabular}{c|c|c}
        \# Agents  & $\vert \x^*-\x_{SOL}\vert$ & \# Iterations \\
        \hline\hline
        10                         & 0.0136             & 2                             \\
        20                         & 0.0050             & 2                             \\
        50                         & 0.0282             & 3                             \\
        100                        & 0.0175             & 3                             \\
        1000                       & 0.0020             & 1                             \\
    \end{tabular}
    \caption{Equidistant starting positions (left) and leftsided starting positions (right). $\x^* \approx 1.5355$.}
    \label{tbl:tabelle}
\end{table}
    We can conclude, that if we use the thresholds, the computation is considerably faster. However the results depend on where the agents might be placed and the number of agents used. Moreover we saw that depending on the case, more or less agents should be used to require a fast and precise solution.

\section{Convergence and error analysis}

In this chapter we consider the $SBGD_{p,q}$ iterations \eqref{eg:2.3} using the \emph{backtracking line search} for determining the step size $h_i^n=h(\x_i^n,\lambda\psi_q(\tilde{m}_i^{n+1}))$ with shrinkage factor $\gamma \in (0,1)$.
To simplify matters, we again assume $p=q=1$.
Moreover we assume that for all agents there exists a bounded region $\Omega$. We do not have apriori a bound on $\Omega$, because lighter agents
are allowed to explore the ambient space with larger step sizes. Hence the footprint of the agents $conv_i\{\x_i^n\}$ may expand beyond its initial
convex hull $conv_i\{\x_i^0\}$.

\subsection{Lower bound on step size}

Consider the class of loss functions $F\in\mathcal{C}^2(\Omega)$ with Lipschitz-bound
\begin{align}
    \label{eg:lipbound}
	\lvert \nabla F(\x)-\nabla F(\y)\rvert \leq L\lvert \x - \y \rvert,\hspace{1cm} \forall \x ,\y \in \Omega.
\end{align}

\begin{lemma}
  \label{lemma1}
The lower bound on the step length is
\begin{align}
	h_i^n\geq\frac{2}{L}\gamma(1-\lambda\psi_q(\tilde{m}_i^{m+1})).
\end{align}
\end{lemma}
\begin{proof}
	By Taylor's theorem and the Lipschitz continuity of $\nabla F$
	\begin{align*}
		F(\x_i^n - h_i^n\nabla F(\x_i^n)) &= F(\x_i^n)- h_i^n\lvert\nabla F(\x_i^n)\rvert^2 + \frac{1}{2}(h_i^n)^2\lvert\nabla F(\x_i^n)\rvert^2 H_F(\x_i^n)\\
		&\leq F(\x_i^n)- h_i^n\lvert\nabla F(\x_i^n)\rvert^2 + \frac{L}{2}(h_i^n)^2\lvert\nabla F(\x_i^n)\rvert^2\\
		&=F(\x_i^n)- (1-\frac{L}{2}h_i^n)h_i^n\lvert\nabla F(\x_i^n)\rvert^2.
	\end{align*}
	Hence, if $h_i^n\leq\frac{2}{L}(1-\lambda\psi_q(\tilde{m}_i^{m+1}))$,
	\begin{align}
        \label{eg:5.1.3}
		F(\x_i^n - h_i^n\nabla F(\x_i^n))\leq F(\x_i^n)- \lambda\psi_q(\tilde{m}_i^{m+1})h_i^n\lvert\nabla F(\x_i^n)\rvert^2.
	\end{align}
	By the \emph{backtracking line search} iterations, the inequality holds for $h_i^n$ but not for $\frac{h_i^n}{\gamma}$. \\
	Therefore
	\begin{align*}
		\frac{h_i^n}{\gamma}\geq\frac{2}{L}(1-\lambda\psi_q(\tilde{m}_i^{m+1})).
	\end{align*}
\end{proof}

\subsection{Convergence to a band of local minima}
Next we discuss the convergence of the SBGD method, which is determined by the time series of SBGD minimizers,
\begin{align*}
  \X^n_-= \x_{i_n}^n, \hspace{1cm} i_n:=\argmin\limits_{i\in J} F(\x_i^n),
\end{align*}
and the time series of its heaviest agents,
\begin{align*}
  \X^n_+= \x_{j_n}^n, \hspace{1cm} j_n:=\argmax\limits_{i\in J} m_i^n.
\end{align*}
The communication of masses leads to a shift of mass from higher ground to the minimizers. When the minimizer attracts enough mass to gain the role of the
heaviest agent, the two sequences coincide.
Furthermore, since $F(\x_i^n)$ are decreasing, we conclude that $\forall n,i$ the SBGD iterations remain in a range
\begin{align}
  \max\limits_{j\in J} F(\x_j^n)-F(\x_i^n)\leq M, \hspace{1cm} M:=\max\limits_{i\in J} F(\x_i^0)-F(\x^*).
\end{align}

\begin{proposition}
    \label{prop}
  Let $\{\X^n_-\}_{n\geq 0}$ and $\{\X^n_+\}_{n\geq 0}$ denote the time sequence of SBGD minimizers and heaviest agents at $t^n$. Then
  there exists a constant $C=C(\gamma,L,M,\lambda)$, such that
  \begin{align}
      \label{eg:5.2.2}
    \sum\limits_{n=0}^\infty \min\{\lvert\nabla F(\X^n_+)\rvert,\lvert\nabla F(\X^n_-)\rvert,\lvert\nabla F(\X^n_+)\rvert\cdot\lvert\nabla F(\X^n_-)\rvert\}^2
    < C \min\limits_{i\in J} F(\x_i^0).
  \end{align}
\end{proposition}
\begin{proof}
    By Lemma \eqref{lemma1} the step length $h_i^n\geq\frac{2\gamma}{L}(1-\lambda\psi_q(\tilde{m}_i^{m+1}))$ and the descent property \eqref{eg:5.1.3} implies
    \begin{align}
        \label{eg:descentproperty}
        \begin{split}
        F(\x_i^{n+1}) &\leq F(\x_i^n)-\lambda \psi_q(\tilde{m}_i^{m+1})h_i^n \vert \nabla F(\x_i^n)\vert ^2 \\
        &\leq F(\x_i^n)-\frac{2\gamma}{L}(1-\lambda\psi_q(\tilde{m}_i^{m+1}))\lambda \psi_q(\tilde{m}_i^{m+1})  \vert \nabla F(\x_i^n)\vert ^2 \\
        &\stackrel{q=1}{=}F(\x_i^n)-\frac{2\gamma}{L}(1-\lambda\tilde{m}_i^{m+1})\lambda \tilde{m}_i^{m+1}  \vert \nabla F(\x_i^n)\vert ^2.
    \end{split}
    \end{align}
    Since the descent bound applies for all agents, we consider that bound for the minimizer $i_n=\argmin\limits_{i\in J} F(\x_i^n)$.
    For this purpose we need to distinguish two scenarios: \\\\
    The first scenario is the canonical scenario, in which the minimizer coincides with the heaviest agent, $m_{i_n}^{n+1}=m_{+}^{n+1}$ and thus $\tilde{m}_{i_n}^{n+1}=1$. Hence, we conclude $h_{i_n}^n\geq\frac{2\gamma}{L}(1-\lambda)$ and because $\X_-^{n+1}=\x_{i_{n+1}}^{n+1}$ is
    the global minimizer at time $t^{n+1}$,
    \begin{align}
        \label{eg:5.2.4}
        F(\X_-^{n+1}) \leq F(\x_{i_n}^{n+1}) \leq F(\X_-^n)-\frac{2\gamma}{L}(1-\lambda)\lambda \vert \nabla F(\X_-^n)\vert ^2
    \end{align}
    holds. The second scenario takes place, when the mass of the minimizer is $m_{i_n}^{n+1}<m_{j_n}^{n+1}$. That means the mass of the minimizer is not greater yet, than the mass of the heaviest agent from the iteration before positioned at $\x_{j_n}^n$. So even after losing a portion of its mass,
    \begin{align*}
        & m_{+}^{n+1} = m_{+}^n-\eta_{j_n}^n m_{+}^n \\
        \Leftrightarrow \hspace{0.5cm} & m_{+}^{n+1}\eta_{j_n}^n = m_{+}^n\eta_{+}^n-(\eta_{+}^n)^2 m_{+}^n \\
        \Leftrightarrow \hspace{0.5cm} & \eta_{+}^n m_{+}^n = \frac{\eta_{+}^n}{(1-\eta_{+}^n)}m_{+}^{n+1}
    \end{align*}
    the heaviest agent from before is still heavier than the minimizer at $\x_{i_n}^{n+1}$. Because the minimizer gained the mass lost by the heaviest agent,
    for the relative mass of the minimizer applies
    \begin{align*}
        m_{i_n}^{n+1}>\eta_{+}^n m_{+}^n = \frac{\eta_{+}^n}{(1-\eta_{+}^n)}m_{+}^{n+1},
    \end{align*}
    and thus with \eqref{eg:relativeheights} we conclude
    \begin{align*}
        \tilde{m}_{i_n}^{n+1} &> \frac{\eta_{+}^n}{(1-\eta_{+}^n)} \\
        & = \frac{F(\x_{+}^n)-F(\x_{i_n}^n)}{\max\limits_{j\in J}F(\x_j^n)-F(\x_{+}^n)} \\
        & > \frac{1}{M}(F(\X_+^n)-F(\X_-^n)).
    \end{align*}
    Depending on $F(\X_+^n)-F(\X_-^n)$, there are two subcases of this scenario: \\
    First we assume $F(\X_+^n)-F(\X_-^n)\leq \frac{\gamma}{L}(1-\lambda)\lambda \vert \nabla F(\X_+^n)\vert ^2$.
    With that and the descent property \eqref{eg:descentproperty} for the heaviest agent at $\x_+^n$, where $\tilde{m}_{i_n}^{n+1}\to\tilde{m}_{+}^{n+1} = 1$, we find
    \begin{align}
        \label{eg:5.2.5}
        \begin{split}
            F(\X_-^{n+1})&\leq F(\X_+^{n+1}) \leq F(\X_+^n)-\frac{2\gamma}{L}(1-\lambda)\lambda \vert \nabla F(\X_+^n)\vert ^2 \\
            & \leq F(\X_-^n) - \frac{\gamma}{L}(1-\lambda)\lambda \vert \nabla F(\X_+^n)\vert ^2 .
        \end{split}
    \end{align}

    We remain with the worst case scenario $F(\X_+^n)-F(\X_-^n)\geq \frac{\gamma}{L}(1-\lambda)\lambda \vert \nabla F(\X_+^n)\vert ^2$. In this case,
    we have a large difference of heights between the minimizer and the heaviest agent. This implies
    \begin{align*}
        \tilde{m}_{i_n}^{n+1}> \frac{1}{M}(F(\X_+^n)-F(\X_-^n)) \geq \frac{\gamma}{ML}(1-\lambda)\lambda \vert \nabla F(\X_+^n)\vert ^2.
    \end{align*}
    Without loss of generality \cite[S. 15]{tad}, we can assume $(1-\lambda\tilde{m}_{i_n}^{n+1})>\frac{1}{2}$. Together with the secured lower bound for $\tilde{m}_{i_n}^{n+1}$
    and the descent property \eqref{eg:descentproperty} we conclude
    \begin{align}
        \label{eg:5.2.6}
        \begin{split}
            F(\X_-^{n+1})&\leq F(\x_{i_n}^{n+1}) \leq F(\x_{i_n}^{n}) - \frac{2\gamma}{L}(1-\lambda\tilde{m}_{i_n}^{m+1})\lambda \tilde{m}_{i_n}^{m+1}  \vert \nabla F(\x_i^n)\vert ^2 \\
            & \leq F(\x_{i_n}^{n}) - \frac{\gamma}{L}\lambda \tilde{m}_{i_n}^{m+1}  \vert \nabla F(\x_i^n)\vert ^2 \\
            &\leq F(\X_-^n) - \frac{\gamma^2}{ML^2}(1-\lambda)\lambda^2 \vert\nabla F(\X_-^n)\vert ^2 \nabla F(\X_+^n)\vert ^2.
        \end{split}
    \end{align}
    \\\\
    By combining all cases \eqref{eg:5.2.4}, \eqref{eg:5.2.5} and \eqref{eg:5.2.6}, we find
    \begin{align}
        F(\X_-^{n+1})\leq F(\X_-^n) - \frac{1}{C}\min\{\lvert\nabla F(\X^n_+)\rvert,\lvert\nabla F(\X^n_-)\rvert,\lvert\nabla F(\X^n_+)\rvert\cdot\lvert\nabla F(\X^n_-)\rvert\}^2
    \end{align}
    and
    \begin{align}
        C = \max\left\{\frac{L}{\gamma(1-\lambda)\lambda},\frac{ML^2}{\gamma^2(1-\lambda)\lambda^2}\right\}.
    \end{align}
    The bound in \eqref{eg:5.2.2} follows by a telescoping sum.
\end{proof}
We have seen that the summability bound \eqref{eg:5.2.2} depends only on the time sequence of minimizers and heaviest agents, but not on the lighter agents.
For $n$ large enough, both sequences coincide into one sequence $\{\X^n\}$. To prove convergence of SBGD, time sub-sequences $\{\X^{n_\alpha}\}$ need to
satisfy a Palais-Smale condition \cite{str}:
by monotonicity $F(\X^{n_\alpha})\leq \max\limits_{i\in J} F(\x_i^0)$ and $\nabla F(\X^{n_\alpha})\overset{\alpha\to\infty}{\longrightarrow} 0$.

\begin{theorem}
    Consider the loss function $F\in\mathcal{C}^2(\Omega)$ such that the Lipschitz-bound \eqref{eg:lipbound} holds and let $\{\X^n_-\}_{n\geq 0}$ denote the
    time sequence of SBGD minimizers. Then $\{\X^n_-\}_{n\geq 0}$ consists of one or more sub-sequences, $\{\X^{n_\alpha}_-,\alpha = 1,2,\hdots\}$, that converge to a band of local minima with equal heights,
    \begin{align*}
        \X_-^{n_\alpha}\overset{n_\alpha\to\infty}{\longrightarrow} \X_\alpha^*
    \end{align*}
    such that $\nabla F(\X_\alpha^*)=0$ and $F(\X_\alpha^*)=F(\X_\beta^*)$. In particular, if $F$ admits only distinct local minima in $\Omega$,
    then the whole sequence $\X^n$ converges to a minimum.
\end{theorem}
\begin{proof}
    Because we assume the sequence $\{\X_-^n\}$ is bounded in $\Omega$, we know it has converging sub-sequences. We take any converging sub-sequences
    $\X_-^{n_\alpha}\to \X_\alpha^* \in \Omega$. Then by the proposition \eqref{prop} before,
    \begin{align*}
        \nabla F(\X^{n_\alpha})\to 0,
    \end{align*}
    and thus $\X_\alpha^*$ are local minimizers with $\nabla F(\X_\alpha^*) = 0$. Since $F(\X_-^n)$ is decreasing, all $F(\X_\alpha^*)$ must have the
    same height. The collection of equi-height minimizers $\{\X_\alpha^*\vert F(\X_\alpha^*)=F(\X_\beta^*)\}$ is the limit-set of $\{\X_-^n\}$.
\end{proof}
\subsection{Flatness and convergence rate}

To quantify convergence rate, we need a classification for the level of flatness our target function has. For this we use the Lojasiewicz condition \cite{toj}: \\
If $F$ is analytic in $\Omega$, then for every critical point of $F$, $\x^*\in\Omega$, exists a neighborhood $\mathcal{N}_*$ surrounding $\x^*\in\mathcal{N}_*$,
an exponent $\beta\in(1,2]$ and a constant $\mu>0$ such that
\begin{align}
    \label{eg:toj}
    \mu\lvert F(\x)-F(\x^*)\rvert\leq\lvert\nabla F(\x)\rvert^\beta, \hspace{1cm} \forall\x\in\mathcal{N}_*.
\end{align}
In case of local convexity, \eqref{eg:toj} is reduced to the Polyack-Lojasiewicz condition \cite{pol}
\begin{align}
    \label{eg:pol}
    \mu ( F(\x)-F(\x^*))\leq\lvert\nabla F(\x)\rvert^2, \hspace{1cm} \forall\x\in\mathcal{N}_*.
\end{align}

For the next theorem we assume that $n$ is large enough, so that we are allowed to discuss only the canonical scenario, where minimizers and heaviest agents coincide.
\begin{theorem}
    Consider the loss function $F\in\mathcal{C}^2$ such that the Lipschitz-bound \eqref{eg:lipbound} holds, with minimal flatness $\beta$.
    Let $\{\X^n_-\}_{n\geq 0}$ denote the time sequence of SBGD minimizers. Then, there exists a constant $C=C(\gamma,\lambda,\mu)$, such that
    \begin{align}
        \label{eg:cases}
        F(\X_-^{n_\alpha})-F(\X_\alpha^*)
        \begin{cases}
            \leq \left(1-\frac{2\mu\gamma\lambda(1-\lambda)}{L}\right)^n (\min\limits_{i\in J} F(\x_i^0)-F(\x^*)), & \beta = 2 \\
            \lesssim C\left(\frac{1}{n_\alpha}\right)^{\frac{\beta}{2-\beta}}, & \beta\in (1,2)
        \end{cases}
    \end{align}
\end{theorem}
\begin{proof}
    We again start with the descent property \eqref{eg:5.2.4}
    \begin{align*}
        F(\X_-^{n+1})\leq F(\X_-^n) - \lambda h_- \lvert\nabla F(\X_-^n)\rvert^2, \hspace{1cm} h_{i_n}^n \geq h_-:=\frac{2\gamma}{L}(1-\lambda).
    \end{align*}
    Focussing on the converging sub-sequence $\{\X_-^{n_\alpha}\}$, we get
    \begin{align*}
        F(\X_-^{n_\alpha+1})\leq F(\X_-^{n_\alpha}) - \lambda h_- \lvert\nabla F(\X_-^{n_\alpha})\rvert^2.
    \end{align*}
    For the quadratic case of the Polyack-Lojasiewicz condition \eqref{eg:pol}, we follow
    \begin{align}
        \begin{split}
         F(\X_-^{n_\alpha+1})&\leq F(\X_-^{n_\alpha}) - \mu\lambda h_- ( F(\X_-^{n_\alpha})-F(\X_\alpha^*)), \hspace{1cm} \X_-^{n_\alpha}\in\mathcal{N}_\alpha \\
        \Leftrightarrow \hspace{0.5cm} F(\X_-^{n_\alpha+1})-F(\X_\alpha^*) &\leq (1-\mu\lambda h_-)(F(\X_-^{n_\alpha})-F(\X_\alpha^*)) \\
        \Leftrightarrow \hspace{0.9cm} F(\X_-^{n_\alpha})-F(\X_\alpha^*) &\leq (1-\mu\lambda h_-)^n(F(\x^0)-F(\X_\alpha^*))\\
            &\leq \left(1-\frac{2\mu\gamma\lambda(1-\lambda)}{L}\right)^n (\min\limits_{i\in J} F(\x_i^0)-F(\x^*)).
        \end{split}
    \end{align}
    Now consider the error $E_{n_\alpha}:=F(\X_-^{n_\alpha})-F(\X_\alpha^*)$ in the case of general Lojasiewicz bound \eqref{eg:toj}, then
    \begin{align*}
        E_{n_{\alpha+1}} \leq E_{n_\alpha}-\lambda h_-(\mu E_{n_\alpha})^{\frac{1}{\beta}}, \hspace{1cm}  \X_-^{n_\alpha}\in\mathcal{N}_\alpha.
    \end{align*}
    This is a Riccati inequality \cite{ric} and the solution yields
    \begin{align*}
        F(\X_-^{n_\alpha})-F(\X_\alpha^*)&\lesssim (\lvert\min\limits_{i\in J} F(\x_i^0)-F(\x^*)\rvert^{-\frac{1}{\beta'}}
        + \lambda h_-\mu^{\frac{2}{\beta}}n_\alpha)^{-\beta'}, \hspace{1cm} \beta' = \frac{\beta}{2-\beta}>1 \\
        & \lesssim \lvert\min\limits_{i\in J} F(\x_i^0)-F(\x^*)\rvert + (\lambda h_-)^{-\beta'}\mu^{\frac{-2\beta'}{\beta}}n_\alpha^{-\beta'} \\
        & \lesssim \underbrace{\left(\frac{L}{2\mu\gamma\lambda(1-\lambda)}\right)^{\frac{\beta}{2-\beta}}\left(\frac{1}{\mu}\right)^{\frac{2}{2-\beta}}}_{C:=}\left(\frac{1}{n_\alpha}\right)^{\frac{\beta}{2-\beta}}
    \end{align*}
\end{proof}

\section{Conclusion and Outlook}
The swarm-based method is a new approach for non-convex optimization. By applying the model of a swarm, we use different agents to find a global minimum.
The swarm-based gradient descent is one method from a class of swarm-based methods. Also based on swarm behavior is the swarm-based random descent (SBRD) method \cite{zen}.
This method uses random descent directions to improve the global swarm position and to find the global minimum.
Therefore, following this work, further insights into swarm-based methods is possible and needed.\\\\

In this thesis we learned about the swarm-based gradient descent, how it works, how we can implement it, and we discussed the convergence rate.
We learned, that the key element of this method is the communication. By communicating the swarm avoids being trapped
in basins of local minima and therefore is able to reach the global minimum, as we saw by an example. Compared to the backtracking gradient descent method,
the SBGD method is less dependent on the initial starting positions. However, the performance of SBGD depends on the number of agents, the parameter $q$ and the used thresholds.
To get the best results, it is necessary  to find a balance between all of them. Although we can say in general, the more agents are used,
the more precise the result will be. \\\\

However, the fine-tuning aspect should be further looked at with more examples. Depending on the case, different options for $q$ and the thresholds might be considered best.
Moreover, it is possible to use other methods, than the \emph{backtracking line search} to compute a step length. The question is, how other methods affect the
SBGD-method and the results.
Furthermore, we need to discuss the SBGD method for higher dimensional functions. In the one-dimensional case we saw advantages of using SBGD compared to other
gradient methods. But does this still apply for the higher dimensional case, or is the SBGD method even more superior for this case?

\nopagebreak
\newpage
\addcontentsline{toc}{section}{Symbols}
\section*{Symbols\label{Notationsverzeichnis}}
\begin{longtable}{p{3.5cm} p{12cm}}
$J $									& Number of agents \\
$\x_i(t)$							& Position of the i-th agent \\
$m_i(t)$							& Mass of the i-th agent \\
$h$								  	& Time step \\
$\tilde{m}_i(t)$			& Relative mass of the i-th agent \\
$m_+$							    & Maximum mass $m_i(t)$ for $i \in J$ \\
$\psi_q (\tilde{m})$	& Influence of relative mass on time step using parameter $q>0$ \\
$F_{\max}(t)$					& Maximum height of swarm at time $t$ \\
$F_{\min}(t)$					& Minimum height of swarm at time $t$  \\
$\eta_i(t)$						& Relative height of the i-th agent \\
$\phi_p(\eta)$				& Degree of mass transition with parameter $p>0$ \\
$\x_+$							  & Worst positioned agent \\
$\gamma$                    & Shrinking factor for backtracking method
\end{longtable}


\newpage
\nocite{*}
 \printbibliography[keyword={quellen}]
\addcontentsline{toc}{section}{References}

\newpage
\addcontentsline{toc}{section}{Eigenständigkeitserklärung}
\section*{Eigenständigkeitserklärung}
Hiermit versichere ich an Eides statt, dass ich die vorliegende Arbeit selbstständig und ohne die Benutzung anderer als der angegebenen Hilfsmittel angefertigt habe. Alle Stellen, die wörtlich oder sinngemäß aus veröffentlichten und nicht veröffentlichten Schriften entnommen wurden, sind als solche kenntlich gemacht.
Die Arbeit ist in gleicher oder ähnlicher Form oder auszugsweise im Rahmen einer anderen Prüfung noch nicht vorgelegt worden. Ich versichere, dass die eingereichte elektronische Fassung der eingereichten Druckfassung vollständig entspricht.
\\\\\\\\
Janina Tikko
\end{document}